\documentclass[11pt,reqno]{amsart}

\usepackage{amsmath,amssymb,amsthm, amsfonts}   
\usepackage{bbm} 
\usepackage{graphicx}   
\usepackage{pstricks}   
\usepackage[colorlinks=true, linkcolor=blue, citecolor=magenta]{hyperref} 
\usepackage{url}
\usepackage{tasks}
\usepackage{enumerate}
\usepackage{caption,subcaption} 

\usepackage{mathtools}      
\mathtoolsset{showonlyrefs} 

\usepackage{mathrsfs}
\usepackage{multicol}   
\usepackage{fullpage}   

\usepackage[foot]{amsaddr}

\renewcommand\thesection{\arabic{section}}
\newtheorem{thm}{Theorem}[section]
\newtheorem{cor}[thm]{Corollary}
\newtheorem{prop}[thm]{Proposition}
\newtheorem{lem}[thm]{Lemma}

\newtheorem*{acknowledgements*}{Acknowledgements}

\theoremstyle{definition}
\newtheorem{defin}[thm]{Definition}

\theoremstyle{remark}
\newtheorem{rem}[thm]{Remark}

\title{Geometric differentiability of Riemann's non-differentiable function.
}
\author{Daniel Eceizabarrena}
\address{BCAM - Basque Center for Applied Mathematics. Alameda de Mazarredo 14, 48009 Bilbao, Spain.  e-mail: {\tt deceizabarrena@bcamath.org}}

\begin{document}

\begin{abstract}
Riemann's non-differentiable function is a classic example of a continuous function which is almost nowhere differentiable, and many results concerning its analytic regularity have been shown so far.
However, it can also be given a geometric interpretation, so questions on its geometric regularity arise. This point of view is developed in the context of the evolution of vortex filaments, modelled by the Vortex Filament Equation or the binormal flow, in which a generalisation of Riemann's function to the complex plane can be regarded as the trajectory of a particle. The objective of this document is to show that the trajectory represented by its image does not have a tangent anywhere. For that, we discuss several concepts of tangent vectors in view of the set's irregularity. 
\end{abstract}

\maketitle



This document follows the expository work \cite{Eceizabarrena2019} and  \cite{Eceizabarrena2019_Part1} by the author. 

\section{Introduction}

According to mathematical folklore, Riemann's non-differentiable function was the first example of a function regular enough to be continuous, yet wildly misbehaved with respect to differentiation. In a conference to the Royal Prussian Academy of Sciences in 1872 \cite{Weierstrass1872}, Weierstrass explained that Riemann had provided it against what was the general belief at the time: that the derivative of a continuous function could be problematic only at certain points. This function was defined as
\begin{equation}\label{RiemannFunctionOriginal}
R(x) = \sum_{n=1}^{\infty}{\frac{ \sin{ (n^2x ) } }{ n^2 }},
\end{equation}
and Riemann conjectured that it was continuous everywhere, but nowhere differentiable. However, he never wrote a proof of this fact. Weierstrass himself regarded it as a tough problem, and the first result was given by Hardy \cite{Hardy1915} half a century later, in 1915, when he stated that $R$ is not differentiable in points $\pi x$ such that $x$ is irrational, and also if $x$ belongs to a certain subset of the rationals. Much later, in the 1970s, Gerver \cite{Gerver1970,Gerver1971} disproved Riemann's conjecture by showing that $R$ was differentiable in $\pi x$ if and only if $x = p/q$ with $p$ and $q$ coprime and both odd integers, and that the value of the derivative at those points is $-1/2$. More recently, further questions of regularity have been tackled \cite{ChamizoCordoba1999,Duistermaat1991,Jaffard1996, JaffardMeyer1996}.

A usual technique consists in analysing the function
\begin{equation}\label{RiemannFunctionDuistermaat}
\phi_D(x) = \sum_{n=1}^{\infty}{ \frac{ e^{i \pi n^2 x} }{ i \pi n^2 }},
\end{equation}
a natural complex generalisation of $R$ satisfying $\pi  \operatorname{Re}\phi_D(x) = R(\pi x) $. In \cite{Duistermaat1991}, Duistermaat computed the asymptotic behaviour of $\phi_D$ around all rational points using Jacobi's theta function and its relationship with the modular group, and a certain selfsimilar behaviour was brought to light. A partial analysis of irrational points was done by means of the continued fraction approximations. In \cite{Jaffard1996}, Jaffard carried on a thorough analysis on the regularity of $R$ and $\phi_D$ in the context of multifractality, showing among other things that 
\begin{equation}\label{SpectrumOfSingularities}
d(\alpha) = 4\alpha - 2, \qquad \forall \alpha \in [1/2,3/4], 
\end{equation} 
where $d(\alpha)$ is the spectrum of singularities of $R$, that is, the Hausdorff dimension of the set of points where the supremum H\"older regularity of $R$ is $\alpha$. This kind of functions with a non-trivial spectrum of singularities are called multifractal. He proved this by means of the wavelet transform and giving a connection between the regularity of $R$ at $x$ and the diophantine properties of the continued fraction approximations to $x$. Also, he established for $R$ the validity of the multifractal formalism proposed by Frisch and Parisi \cite{FrischParisi1985}.

Riemann's function was designed as an example showing regularity pathologies and defined in a completely analytic way, so one should not be surprised by the fact that it has traditionally been studied from an analytic point of view. Nevertheless, it was shown in \cite{DeLaHozVega2014} that it plays a surprising geometric role in the context of the binormal flow, a model for one vortex filament dynamics described by the equation
\[ \boldsymbol{X}_t = \boldsymbol{X}_s \times \boldsymbol{X}_{ss}, \qquad \text{ or equivalently } \qquad \boldsymbol{X}_t  = \kappa \, \boldsymbol{B} \]
for a curve in space $\boldsymbol{X}: \mathbb{R}^2 \to \mathbb{R}^3$ parametrised in arclength $s$ and in time $t$, with curvature $\kappa = \kappa(s,t)$, binormal vector $\boldsymbol{B} = \boldsymbol{B}(s,t)$ and a given initial condition $\boldsymbol{X}(s,0)$. This equation is often referred to as the Vortex Filament Equation (VFE). They analysed the evolution $\boldsymbol{X}_M$ of a planar regular polygon of $M \in \mathbb{N}$ sides, and they followed the time trajectory of a corner, which is a curve in space. They showed numerical evidence that, once this trajectory is properly rescaled, it converges to a plane curve when $M \to \infty$. More precisely, assuming the usual identification $\mathbb{C} \simeq \mathbb{R}^2$, this curve is
\begin{equation}\label{RiemannFunction}
\phi(x) = \sum_{k \in \mathbb{Z}}{ \frac{e^{-4\pi^2 i k^2 x}-1}{-4\pi^2k^2} }, 
\end{equation} 
one more possible generalisation of $R$, since one can check that 
\begin{equation}\label{RelationshipDuistermaatUs}
\phi(x) = -\frac{i}{2\pi}\phi_D(-4\pi x) + i x + \frac{1}{12}.
\end{equation} 
Difficulties to prove this convergence analytically arise because the problem is translated from the VFE to the Nonlinear Schr\"odinger equation, thanks to the so called Hasimoto transformation \cite{Hasimoto1972}
\begin{equation}\label{Hasimoto}
\psi(s,t) = \kappa(s,t) \, e^{i\int_0^s{\tau(\sigma,t)\,d\sigma}},
\end{equation} 
where $\tau$ is the torsion of $\boldsymbol{X}$. Indeed, $\psi$ satisfies 
\begin{equation}\label{NLS}
\partial_t\psi = i\left( \partial_s^2\psi + \frac12\left( \left| \psi \right|^2 + A(t) \right)\psi \right) 
\end{equation} 
for some function $A(t) \in \mathbb{R}$. According to \eqref{Hasimoto}, computing a solution for \eqref{NLS} with the corresponding initial datum $\psi(s,0)$ amounts to obtaining the curvature and the torsion of $\boldsymbol{X}$. If they are non-vanishing and smooth enough, they determine the curve up to rigid motions by the Frenet-Serret formulas, so that $\boldsymbol{X}$ could be recovered.

A reasonable way to model the initial curvature of the $M$-sided polygon is to place $M$ equidistributed Dirac deltas in the interval $[0,2\pi]$ multiplied by coefficients according to the Gauss-Bonnet theorem, and to extend it periodically to the real line. Then, since a planar curve has null torsion, according to \eqref{Hasimoto} we may work with the initial condition
\begin{equation}\label{InitialDatumM}
\psi_M(s,0) = \frac{2\pi}{M}\sum_{k \in \mathbb{Z}}{\delta\left(s - \frac{2\pi }{M}\, k\right)}.
\end{equation} 
However, the Frenet-Serret frame for curves with vanishing curvature may not be well-defined, so the parallel frame is used instead, where the normal plane is given an alternative basis $\{\boldsymbol{e_1},\boldsymbol{e_2}\}$ such that the derivatives of these vectors do not depend on themselves but only on the tangent. Thus, since $\boldsymbol{T}$ is determined by the curve, the evolution of the frame is well-defined as long as the tangent is well-defined. The equations analogue to the Frenet-Serret system for this frame are
\begin{equation}\label{ParallelFrame}
\left(
\begin{array}{c}
\boldsymbol{T} \\
\boldsymbol{e_1} \\
\boldsymbol{e_2}
\end{array}
\right)_s
=
\left( 
\begin{array}{ccc}
0 & \alpha & \beta \\
-\alpha & 0 & 0 \\
-\beta & 0 & 0
\end{array}
\right)
\,
\left(
 \begin{array}{c}
\boldsymbol{T} \\
\boldsymbol{e_1} \\
\boldsymbol{e_2}
\end{array}
\right).
\end{equation}
Here, $\alpha$ and $\beta$ are functions of $s$ and $t$ that adapt perfectly to the setting of the Hasimoto transformation \eqref{Hasimoto}, since one can check that $\psi = \alpha + i\beta$. Therefore, first solving the equation \eqref{NLS} with the datum \eqref{InitialDatumM} and then integrating the system \eqref{ParallelFrame}, one should be able to recover $\boldsymbol{T}$ and $\boldsymbol{X}$. The trajectory we seek is $\boldsymbol{X}_M(0,t)$ because $s=0$ corresponds to a corner in the initial datum \eqref{InitialDatumM}.

 We may try to approximate the solution by working with the free Schr\"odinger equation instead of the nonlinear system \eqref{NLS}-\eqref{InitialDatumM}. We may also rescale the initial datum to make it independent of $M$. Hence, the system we have to deal with is
\begin{equation}\label{FreeSchrodingerEquation}
\partial_t\psi = i \, \partial_s^2\psi, \qquad \qquad \psi(s,0) = \sum_{k \in \mathbb{Z}}\delta(s-k).
\end{equation}
To obtain the linear trajectory corresponding to $\boldsymbol{X}_M(0,t)$, heuristics in parallelism to the nonlinear setting suggest to integrate the solution of \eqref{FreeSchrodingerEquation} twice in space, or once in time, to get
\begin{equation}\label{DerivationOfPhi}
 i \int_0^t{ \psi(0,t')\, dt' } = i \int_0^t{ e^{i t' \partial_s^2}\left( \sum_{k \in \mathbb{Z}}\delta(\cdot - k) \right)(0)  \,dt' } = i \int_0^t { \sum_{k \in \mathbb{Z}}{ e^{-4\pi^2 i k^2 t'}   }  \,dt' } = \phi(t),
\end{equation}
which is precisely the proposed version of Riemann's function \eqref{RiemannFunction}. The shape of this trajectory is shown in 
Figure~\ref{FigureCurve}. It was shown numerically in \cite{DeLaHozVega2014} that this approximation, though very crude a priori, adapts very well to reality. Indeed, after rescaling, they show evidence that $\lim_{M \to \infty}\boldsymbol{X}_M(0,t) = \phi(t)$. 

As a curiosity, this model grasps the Talbot effect described in optics \cite{Talbot1836} in a different setting than the original one. Indeed, if instead of space, time is fixed at a rational $p/q$, the curve obtained is a skew-polygon of generally $Mq$ sides. This is also directly related to the axis-switching phenomenon observed in several experiments with non-circular jets, and in particular when working with nozzles with corners, such as of triangular or rectangular shape (see, for instance, \cite{GutmarkGrinstein1999}). 

Everything explained so far suggests that $\phi$ is the natural version of Riemann's function to analyse geometrically, rather than $\phi_D$. Indeed, even if in view of \eqref{RelationshipDuistermaatUs} qualitative analytic results for $\phi_D$ are valid for $\phi$ and vice versa, both functions describe substantially different geometric objects, the linear term on the right in \eqref{RelationshipDuistermaatUs} playing an important role (See Figure~\ref{FigureCurves}, where Figure~\ref{FigureCurveDuistermaat} corresponds to $\phi(x)-ix$). Moreover, both functions being non-injective, it is difficult to measure its effect in the images. 

\begin{figure}[h]
\centering
\begin{subfigure}{0.475\textwidth}
  \centering
  \includegraphics[width=1\linewidth]{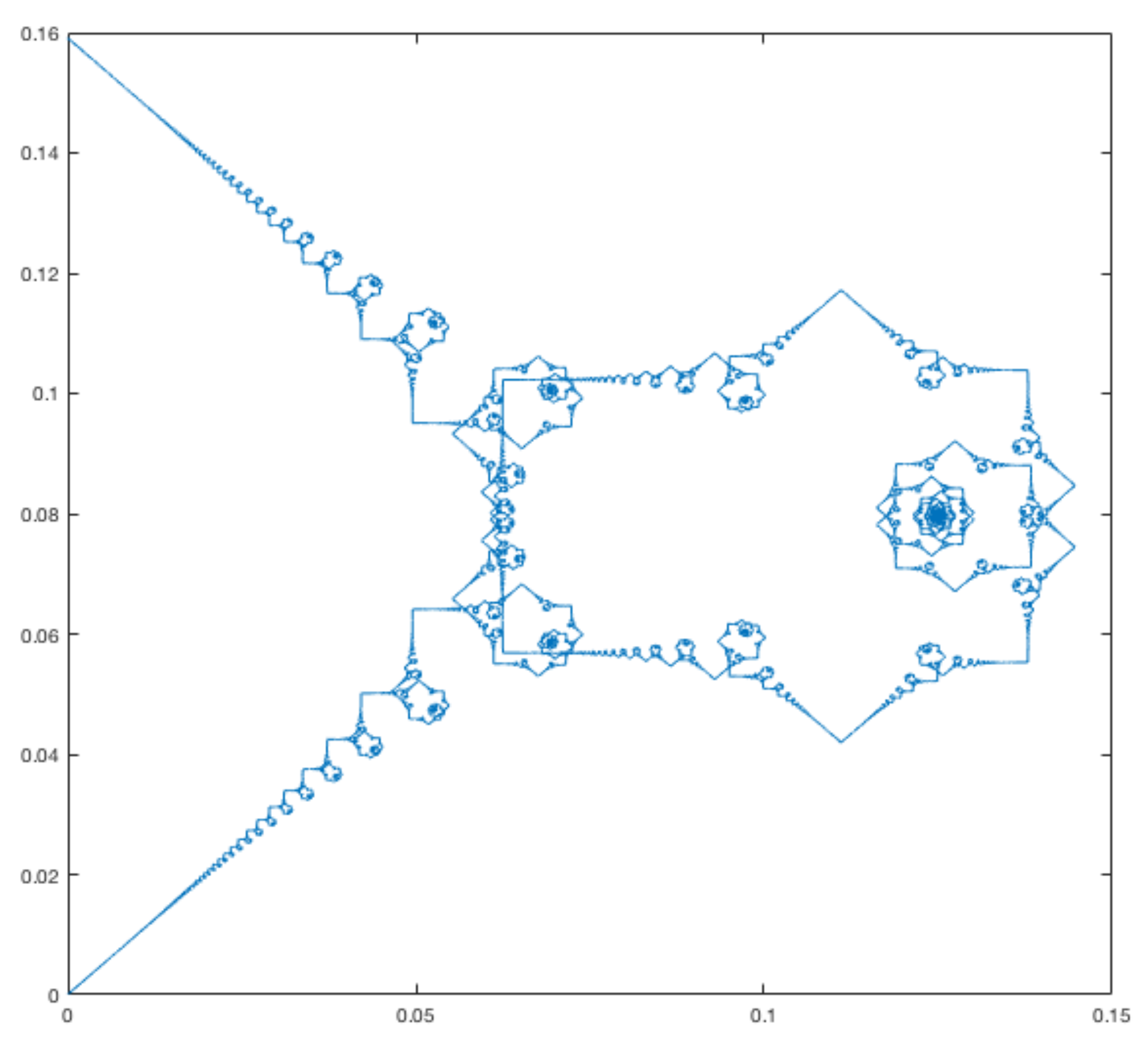}
  \caption{ $ \phi(x) $  }
  \label{FigureCurve}
\end{subfigure}%
\begin{subfigure}{0.05\textwidth}
  \centering
  \hspace{1pt}
\end{subfigure}%
\begin{subfigure}{0.475\textwidth}
  \centering
  \includegraphics[width=0.99\linewidth]{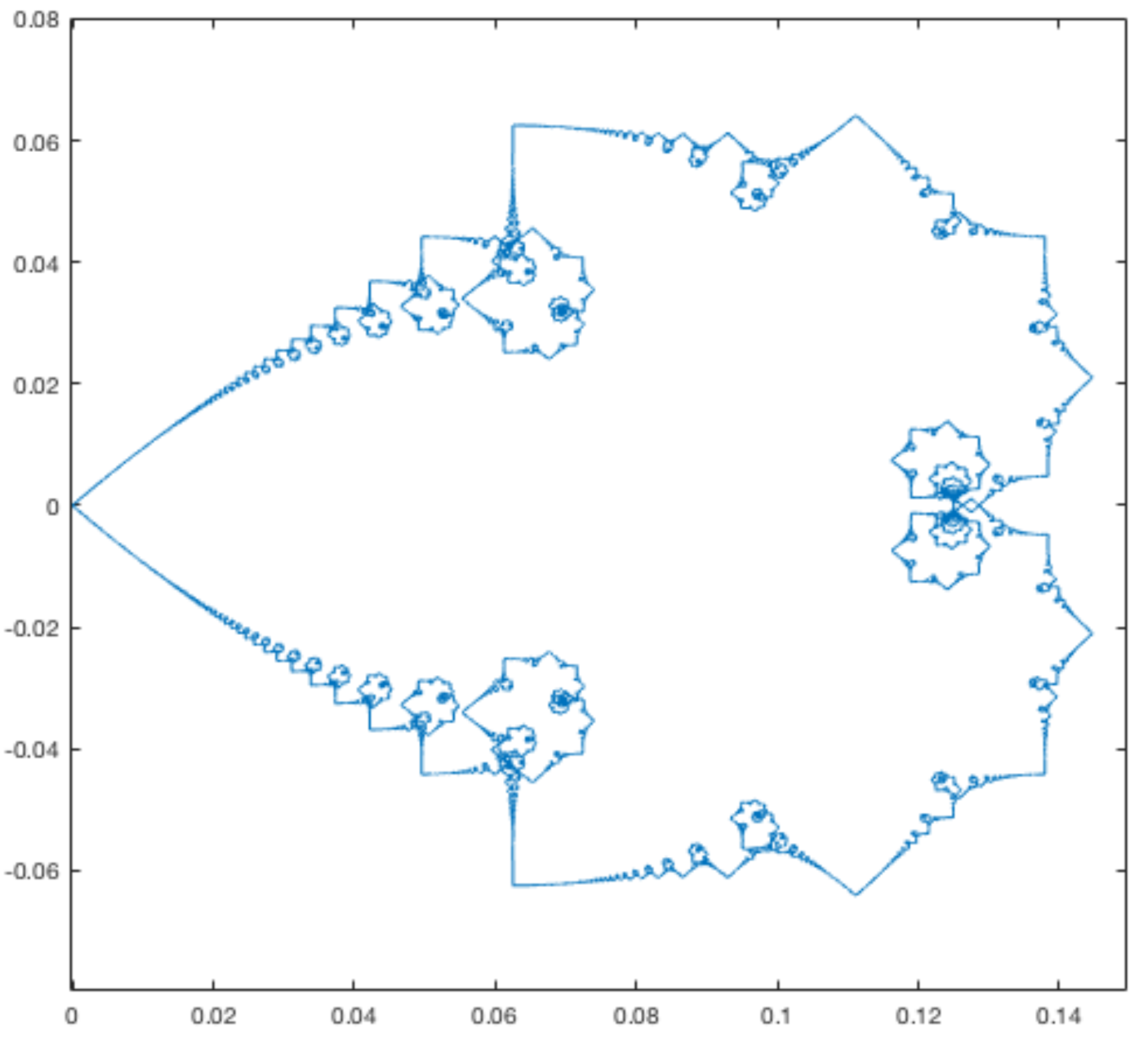}
  \caption{ $ -\frac{i}{2\pi}\phi_D(-4\pi x) + \frac{1}{12} $ }
  \label{FigureCurveDuistermaat}
\end{subfigure}
\caption{The images of the curves in the period $x \in [0,1/2\pi]$ as subsets of the complex plane.
}
\label{FigureCurves}
\end{figure}


Inspired by the works above, especially by \cite{Duistermaat1991}, in \cite{Eceizabarrena2019_Part1} the asymptotic behaviour for $\phi$ is given in all points $t_x = x/2\pi$ with $x$ rational. There, a self-similar structure, already visible in Figure~\ref{FigureCurves},  is analytically shown. As a consequence, the first result on the geometry of $\phi(\mathbb{R})$ is proved in terms of a bound for its Hausdorff dimension, which is
\[ 1 \leq \dim_{\mathcal{H}}\phi(\mathbb{R}) \leq 4/3. \]
In the present document, inspired by the question of which the direction a particle could be in the trajectory of the VFE experiment, we analyse the existence of tangents of $\phi(\mathbb{R})$ from a geometric point of view. The main result in this paper, stated here with no technicalities, is the following.
\begin{thm}\label{TheoremIntro}
Let $\phi$ be Riemann's non-differentiable function defined in \eqref{RiemannFunction}.  There does not exist a point in which $\phi(\mathbb{R})$ has a tangent. 
\end{thm}

This result is meaningful because there is not a clear connection between analytic regularity of $\phi$ and geometric tangency to $\phi(\mathbb{R})$. First, the direct analytic approach to tangency is very restricted by the fact that the derivative exists almost nowhere. One could think of checking the points where the derivative exists, but the results of Gerver \cite{Gerver1970,Gerver1971} and \eqref{RelationshipDuistermaatUs} show that the corresponding value of the derivative of $\phi$ is 0, which is useless to determine a tangent. What is more, one can deduce from the asymptotics in \cite{Eceizabarrena2019_Part1} that a spiralling pattern is generated (Figure~\ref{FigureCurve12}). Second, the asymptotics in the rest of rationals show that derivative does not exist and that the regularity is $C^{1/2}$, but nevertheless, they suggest the existence of two different geometric tangents at each side (Figure~\ref{FigureCurve18}). It is the non-matching of both side-tangents which prevents a single tangent to exist.
 
\begin{figure}[h]
\centering
\begin{subfigure}{0.475\textwidth}
  \centering
  \includegraphics[width=1\linewidth]{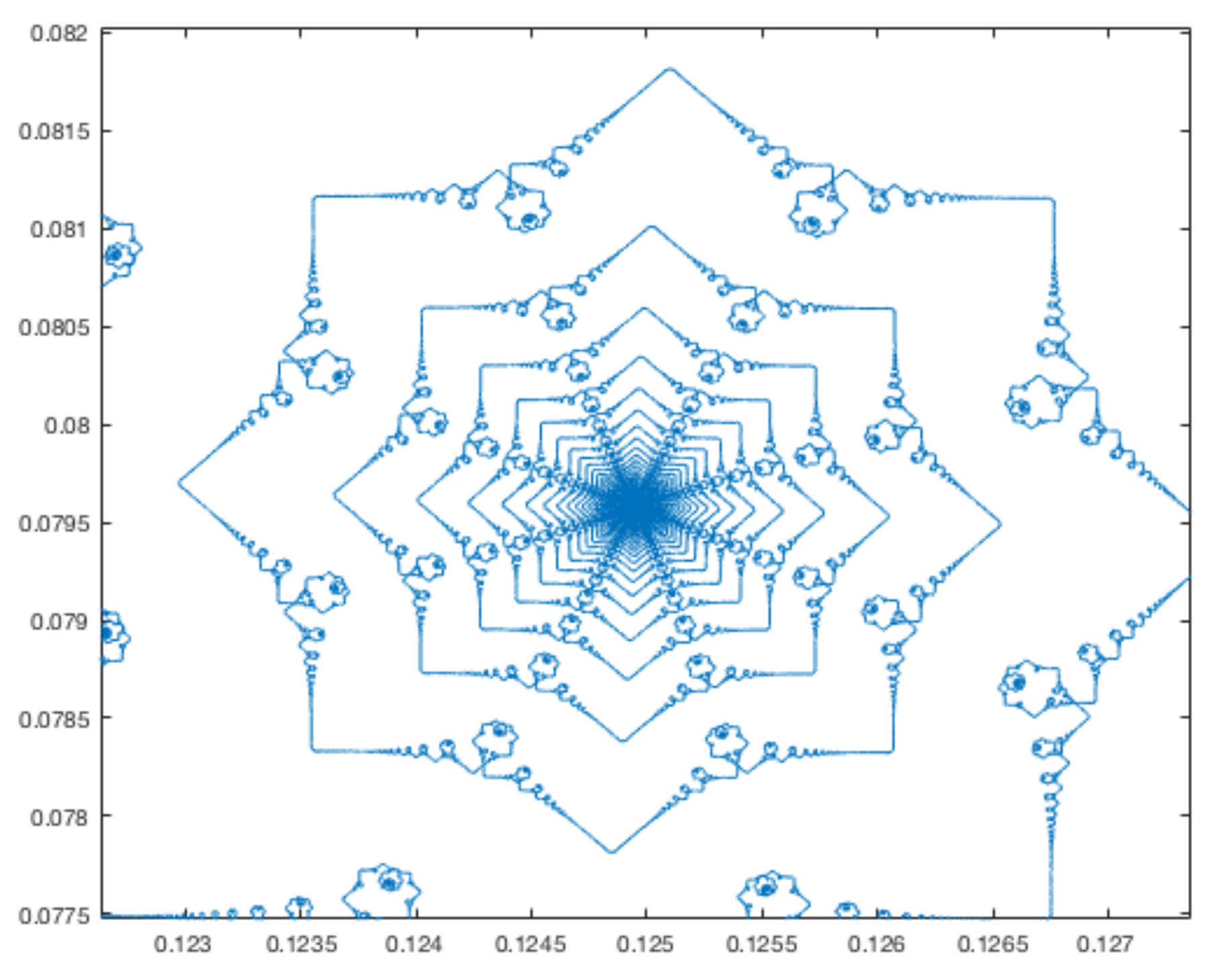}
  \caption{Around $\phi(t_{1/2})$, placed in the centre of the spiral, where precisely the spiral pattern prevents a tangent from forming.}
  \label{FigureCurve12}
\end{subfigure}%
\begin{subfigure}{0.05\textwidth}
  \centering
  \hspace{1pt}
\end{subfigure}%
\begin{subfigure}{0.475\textwidth}
  \centering
  \includegraphics[width=1\linewidth]{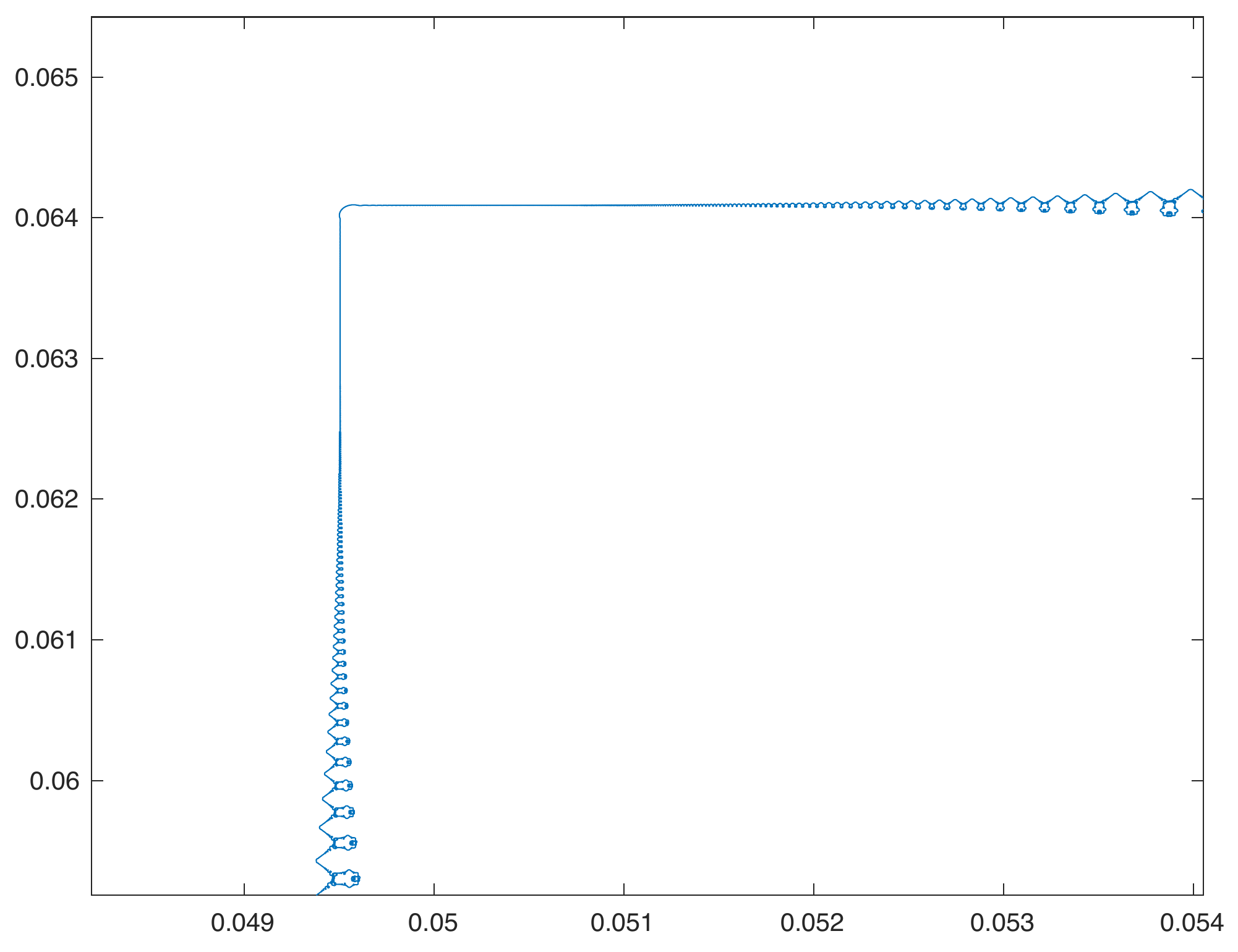}
  \caption{Around $\phi(t_{1/8})$, placed on the corner, where two different side tangents can be distinguished.
  \vspace{0.2cm}
  }
  \label{FigureCurve18}
\end{subfigure}
\caption{Zooms of Figure~\ref{FigureCurve} around the two different types of rationals. 
}
\label{FigureZooms}
\end{figure}

 Last, even if $\phi$ is not differentiable in irrational points, in view of Figure~\ref{FigureCurve18} one cannot directly conclude that there is no geometric tangent there. Also, no asymptotic behaviour around irrationals is available. In an attempt to visualise the situation, in Figure~\ref{FigureZoomsIrrational} we plot the image of  $(t_\rho - \epsilon,t_\rho + \epsilon)$ for an irrational $\rho \in \mathbb{R}\setminus\mathbb{Q}$ and $\epsilon >0$. But instead of seeing a precise behaviour of the function around $t_\rho$, we observe a pattern like in Figure~\ref{FigureZooms} which corresponds to the rational approximation of $\rho$ with smallest denominator in that interval.  However, the fact this pattern changes very much when $\epsilon$ decreases suggests that the behaviour of $\phi$ around $t_\rho$ highly depends on the scale under consideration and that therefore a tangent may not exist. 

\begin{figure}[h]
\centering
\begin{subfigure}{0.475\textwidth}
  \centering
  \includegraphics[width=1\linewidth]{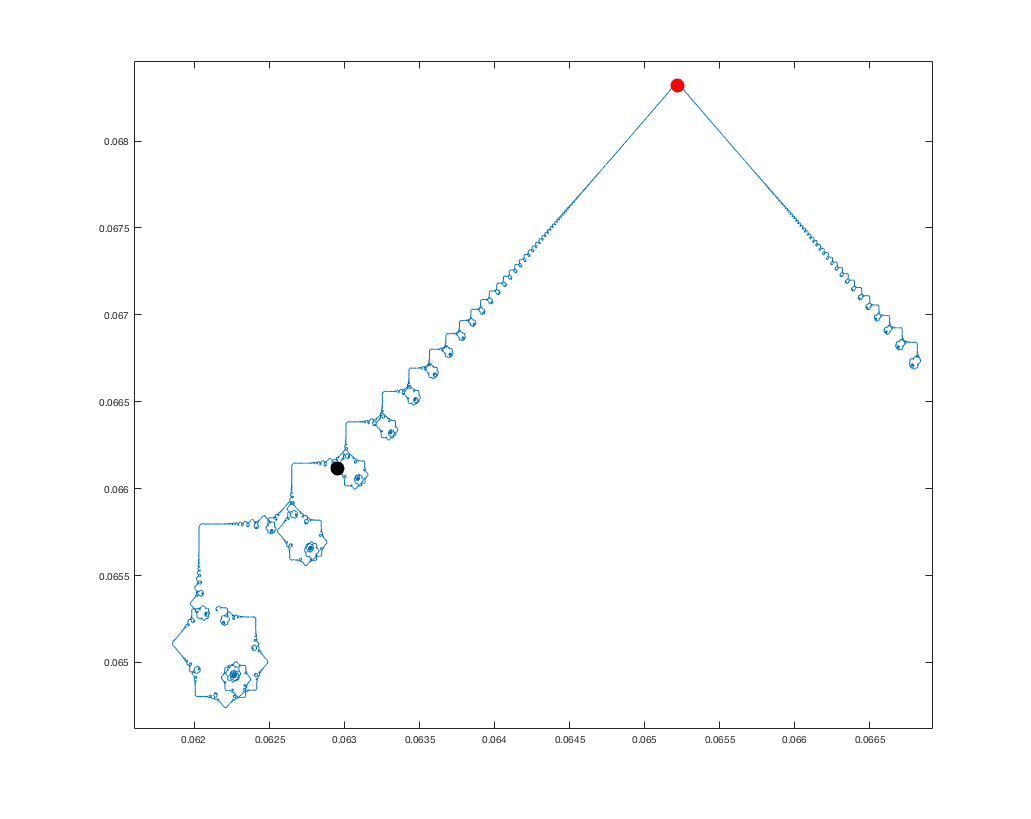}
  \caption{$\epsilon = 0.002$ and the approximation is $\phi(t_{1/7})$. The first approximation to $\pi$ by continued fractions is $22/7$, which for $\pi-3$ turns into $1/7$.
  \vspace{1.2cm}
  }
\end{subfigure}%
\begin{subfigure}{0.05\textwidth}
  \centering
  \hspace{1pt}
\end{subfigure}%
\begin{subfigure}{0.475\textwidth}
  \centering
  \includegraphics[width=1\linewidth]{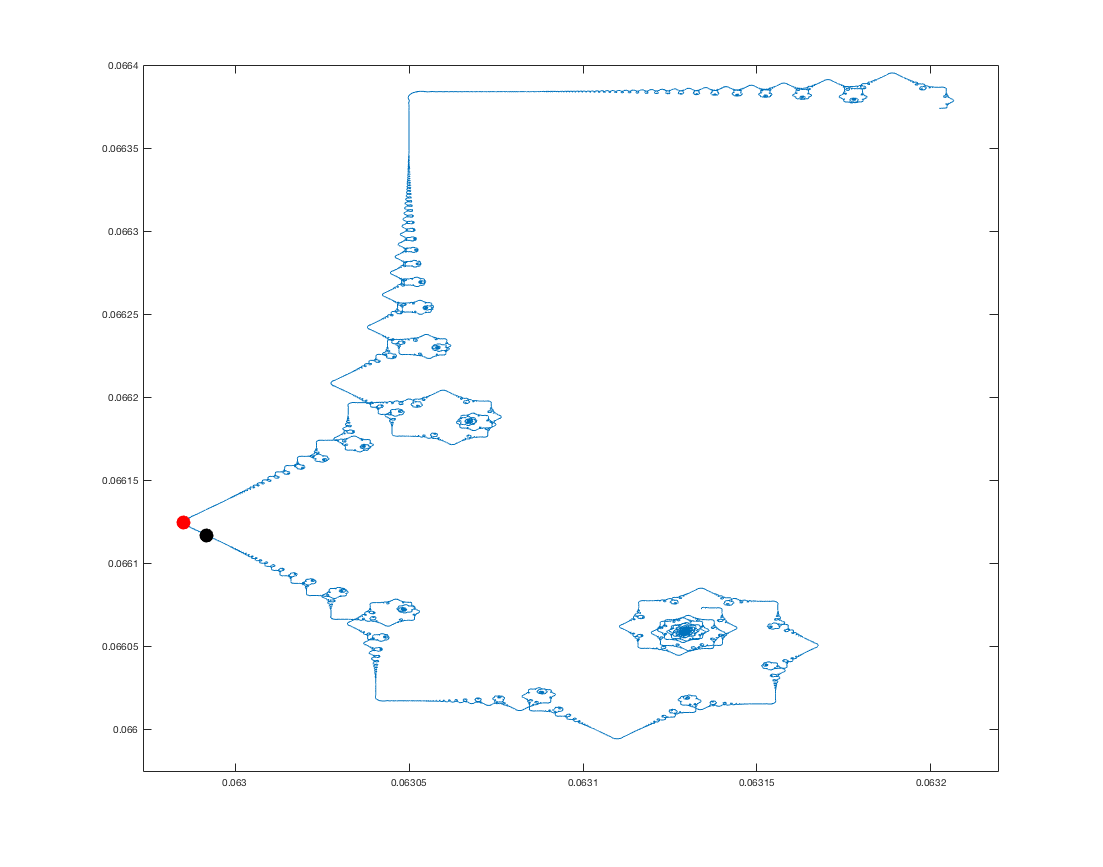}
  \caption{$\epsilon = 0.0001$ and the approximation is $\phi(t_{16/113})$. The third approximation to $\pi$ by continued fractions is the very famous $355/113$, which for $\pi-3$ turns into $16/113$. The spiral on the right corresponds to the second approximation, 333/106, or 15/106 for $\pi-3$. }
\end{subfigure}
\caption{Plots of $\phi((t_{\pi-3} - \epsilon, t_{\pi-3} + \epsilon))$. The black point corresponds to $\phi(t_{\pi-3})$, while the red points correspond to rational approximations which show either a corner or a spiral. Having changed the scale, the behaviour of $\phi$ is completely different. }
\label{FigureZoomsIrrational}
\end{figure}

As already suggested, the main ingredient in the proof of Theorem~\ref{TheoremIntro} is the asymptotic behaviour of $\phi$ around every rational $t_{p/q}$ proved in \cite{Eceizabarrena2019_Part1}, which for convenience we state in the forthcoming sections. Hence, tangency around a rational is easy to manage, but a more subtle analysis is required around an irrational. As we said, no asymptotic behaviour is available around it. A way to know how the function approaches it is to work with rational approximations, among which the approximations by continued fractions are the most effective ones. The proof will then depend on how fast they approach the irrational. This classification was remarked to be important by Jaffard \cite{Jaffard1996} in his multifractal analysis when computing \eqref{SpectrumOfSingularities} and also by Kapitanski and Rodnianski \cite{KapitanskiRodnianski1999} when they studied the regularity of the solution of \eqref{FreeSchrodingerEquation} in the variable $s$ for every fixed time $t$.

As importantly, in this non-canonical setting it is crucial to choose the concept of tangent carefully. Even if $\phi$ is a curve, the classic theory of differential geometry is of no use because Riemann's function is not differentiable in any open set. For this reason, we work with a purely geometric definition coming from geometric measure theory, as well as with one using the parametrisation. The first one is convenient in terms of the irregularity of the set, while the second one allows to perform computations using the asymptotic behaviour. 
An important part of this text is devoted to determine the relationship between them.   

The article is organised as follows. In Section~\ref{SectionTangent}, we introduce two definitions for a tangent vector and work on the relationship between them. We also present the more precise Theorem~\ref{TheoremOfTangents}, which implies Theorem~\ref{TheoremIntro}. The following sections are devoted to prove Theorem~\ref{TheoremOfTangents}: in Section~\ref{Section013} we prove the case of rational points where $\phi$ is not differentiable, and in Section~\ref{Section2}, the case of rational points where $\phi$ is differentiable. Finally, in Section~\ref{SectionIrrationals}, we prove the result in irrational points. 

\begin{acknowledgements*}
Many thanks to Albert Mas and Xavier Tolsa for helpful discussions on different concepts for a tangent, and also to Valeria Banica and Luis Vega for their help and advice.

This research is supported by the Ministry of Education, Culture and Sport (Spain) under grant  FPU15/03078 - Formaci\'on de Profesorado Universitario, by the ERCEA under the Advanced Grant 2014 669689 - HADE and also by the Basque Government through the BERC 2018-2021 program and by the Ministry of Science, Innovation and Universities: BCAM Severo Ochoa accreditation SEV-2017-0718.
\end{acknowledgements*}

\section{A proper way to define a tangent}\label{SectionTangent}

Along this section, we denote by $\mathcal{H}^s$ the $s$-Hausdorff measure, by $\operatorname{dim}_{\mathcal{H}}$ the Hausdorff dimension and by $B(x,r)$ the open ball with center at $x$ and radius $r$.

\subsection{What is a tangent?}

When treating irregular objects in $\mathbb{R}^n$, especially curve-like objects, the question whether it has a tangent is natural. However, a general set need not be the image of a function, let alone be parametrised by a continuous function. A general definition should therefore have a geometric flavour.

Any definition of a tangent at a point should reflect the fact that close to the point, the set is concentrated in a particular direction. Many different approaches to measure this concentration have been proposed. Here, we reproduce the definition given by Falconer for $s$-sets in \cite{Falconer2014}. Let $0 \leq s \leq n$. A Borel set $F\subset \mathbb{R}^n$ is said to be an $s$-set if $\dim_{\mathcal{H}}{F} = s$ and if $0 < \mathcal{H}^s(F) < \infty$. For $x \in \mathbb{R}^n$, $\mathbb{V} \in \mathbb{S}^{n-1}$ and $\varphi > 0$, we define $S_D(x,\mathbb{V}, \varphi)$ to be the closed double cone with vertex $x$, direction $\mathbb{V}$ and opening $\varphi>0$. More precisely, it is the closure of the set consisting of those $y \in \mathbb{R}^n$ such that the vector $y-x$ forms an angle at most $\varphi/2$ with $\mathbb{V}$ or $-\mathbb{V}$ (see figure~\ref{FigureConeGeneral}).
\begin{figure}[h]
\centering
\begin{subfigure}{0.475\textwidth}
  \centering
  \includegraphics[width=1\linewidth]{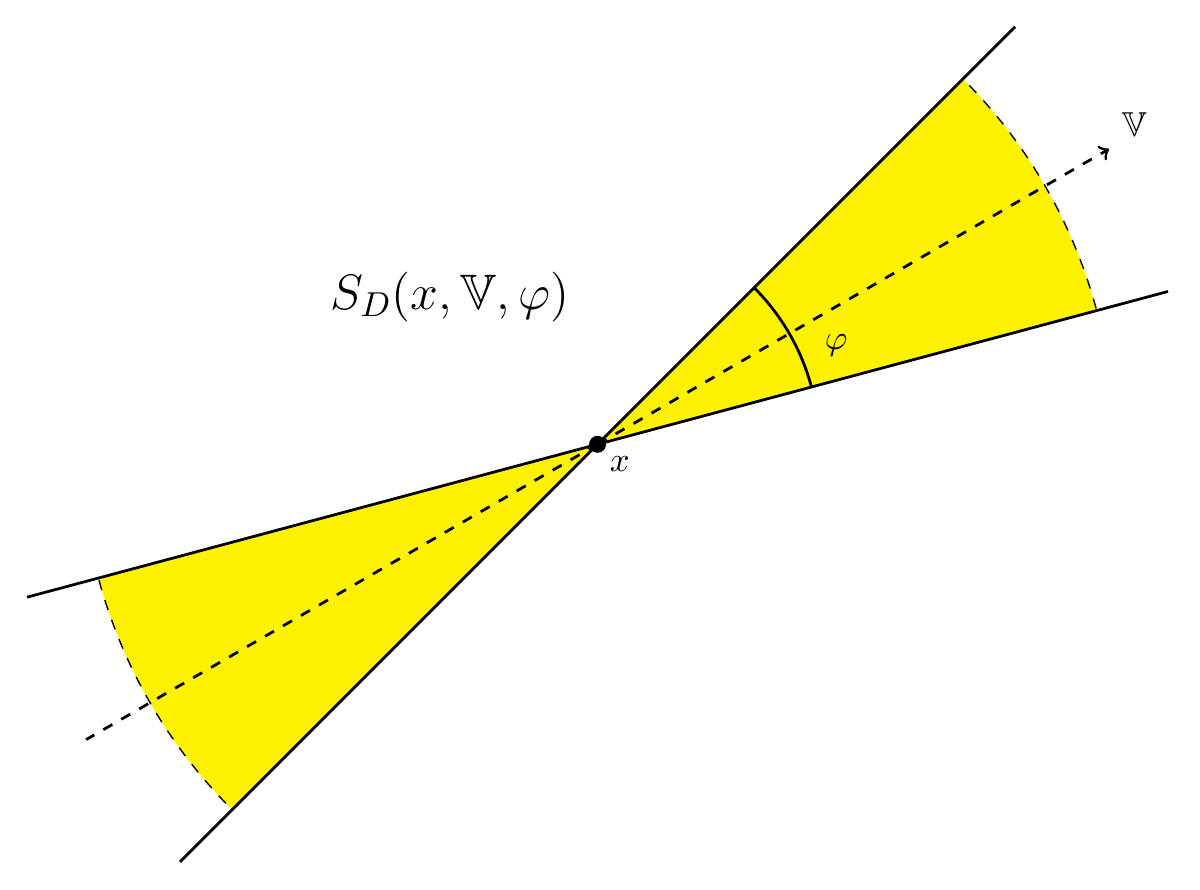}
  \caption{The double cone $S_D(x,\mathbb{V},\varphi)$.}
  \label{FigureConeGeneral}
\end{subfigure}%
\begin{subfigure}{0.05\textwidth}
  \centering
  \hspace{1pt}
\end{subfigure}%
\begin{subfigure}{0.475\textwidth}
  \centering
  \includegraphics[width=1\linewidth]{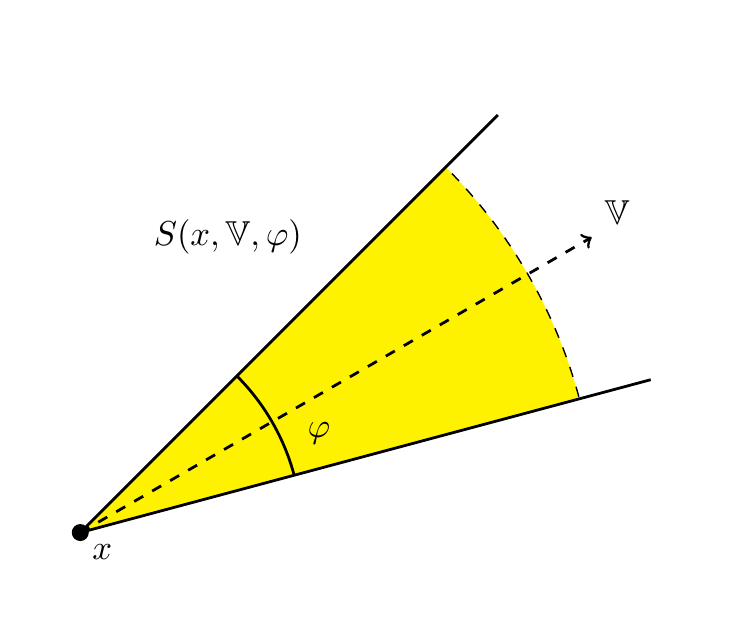}
  \caption{The single cone $S(x,\mathbb{V},\varphi)$.}
  \label{FigureConeSingle}
\end{subfigure}
\caption{Cones.
}
\label{FigureCones}
\end{figure}

\begin{defin}\label{DefFalconer}
Let $0 \leq s \leq n$ and $F \subset \mathbb{R}^n$ be an $s$-set. We say that $\mathbb{V} \in \mathbb{S}^{n-1}$ is a tangent of $F$ at $x \in F$ if
\begin{equation}\label{TangentCond1}
\overline{D}^s(F,x) = \limsup_{r \to 0}{ \frac{\mathcal{H}^s(F \cap B(x,r))}{(2r)^s} } > 0
\end{equation}  
and if for every $\varphi > 0$, 
\begin{equation}\label{TangentCond2}
\lim_{r \to 0}{\frac{\mathcal{H}^s(F \cap (B(x,r) \setminus S_D(x,\mathbb{V},\varphi)))}{(2r)^s}} = 0.
\end{equation}  
\end{defin}
Condition \eqref{TangentCond1} means that there is some concentration of the set $F$ around $x$, no matter how close we are from it. Since the cones in condition \eqref{TangentCond2} can be as narrow as we wish, we also ask that this concentration only happens in direction $\mathbb{V}$.

However, this definition requires the knowledge of the Hausdorff dimension of the set. If $F = \phi(\mathbb{R})$, we only know that $1 \leq \operatorname{dim}_{\mathcal{H}}F \leq 4/3$, so there is no obvious choice of $s$ to use in Definition~\ref{DefFalconer}. On the other hand, when working with a curve, it is natural to use some one dimensional measure such as the length, even if there are curves with infinite length and even some whose images have Hausdorff dimension greater than 1. Therefore, we propose an alternative one-dimensional approach by means of the 1-Hausdorff content, defined as
\[ \mathcal{H}^1_{\infty}(F) = \inf \left\{ \sum_{i\in I}{\operatorname{diam}U_i} : F \subset \bigcup_{i\in I}{U_i},  \,\,\,  I  \text{ countable } \right\}, \]
where the sets $U_i$ can be chosen to be open sets. 

\begin{defin}\label{DefGeometricTangentForCurves}
Let $F \subset \mathbb{R}^n$ be the image of some continuous curve. We say that $\mathbb{V} \in \mathbb{S}^{n-1}$ is a tangent of $F$ at $x \in F$ if
\begin{equation}\label{TangentCond1Content}
\limsup_{r \to 0}{ \frac{\mathcal{H}^1_{\infty}(F \cap B(x,r))}{2r} } > 0
\end{equation}  
and if
\begin{equation}\label{TangentCond2Content}
 \lim_{r \to 0}{\frac{\mathcal{H}^1_{\infty}\left( (F \cap B(x,r)) \setminus S_D(x,\mathbb{V},\varphi) \right)}{2r}} = 0, \qquad \qquad \forall \varphi >0.
\end{equation}  
\end{defin}
\begin{rem}\label{RemarkGeometricDefinition}
Since $\mathcal{H}^1(F) \geq \mathcal{H}^1_\infty(F)$ for any set $F \subset \mathbb{R}^n$,  condition \eqref{TangentCond2} for $s=1$ implies \eqref{TangentCond2Content}. Thus, if for a given $x\in F$ no vector satisfies \eqref{TangentCond2Content} so that no tangent exists in the sense of Definition~\ref{DefGeometricTangentForCurves}, then no tangent exists in the sense of Definition~\ref{DefFalconer} for $s=1$.
\end{rem}

The advantage of using the 1-Hausdorff content is that it is particularly easy to manage in the case of curves, as shown in the following lemma.
\begin{lem}\label{LemmaHausdorffContent}
Let $f:\mathbb{R} \to \mathbb{R}^n$ be a continuous curve and $a,b \in \mathbb{R}$ such that $a < b$. Then,
\[ \mathcal{H}^1_{\infty}(f([a,b])) = \operatorname{diam}( f([a,b]) ). \]
Also, let $x \in \mathbb{R}$, $r>0$ and assume there exists $\epsilon > 0$ such that either $f(x+\epsilon) \notin B(f(x),r)$ or $f(x-\epsilon) \notin B(f(x),r)$. Then,
\[  r \leq  \mathcal{H}^1_{\infty}\left(f(\mathbb{R}) \cap B(f(x),r)\right) \leq 2r. \]
\end{lem}
\begin{proof}
See Subsection~\ref{ProofOfLemmas}.
\end{proof}
An immediate consequence of Lemma~\ref{LemmaHausdorffContent} is that condition~\eqref{TangentCond1Content} is redundant.
\begin{lem}
Let $f: \mathbb{R} \to \mathbb{R}^n$ be a continuous, non-constant curve and $F = f(\mathbb{R})$. Then, for every $x \in \mathbb{R}$ and for small enough $r > 0$, 
\[ \frac12 \leq \frac{ \mathcal{H}^1_{\infty}( F \cap B(f(x),r) ) }{ 2r } \leq 1. \]
\end{lem}
\begin{proof}
Let $x \in \mathbb{R}$. Since $f$ is not a constant, the exists $y \in \mathbb{R}$ such that $f(x) \neq f(y)$. Call $r_x = |f(y) - f(x) | > 0$, so $f(y) \notin B(f(x),r)$ for every $r < r_x$. By the second part of Lemma~\ref{LemmaHausdorffContent} we get
\[ r \leq \mathcal{H}^1_{\infty}\left( f(\mathbb{R}) \cap B(f(x),r) \right) \leq 2r, \qquad \forall r < r_x. \]
\end{proof}

\begin{rem}
One could also state Definition~\ref{DefGeometricTangentForCurves} with the $s$-Hausdorff content
\[ \mathcal{H}^s_{\infty}(F) = \inf \left\{ \sum_{i\in I}{\left(\operatorname{diam}U_i\right)^s} : F \subset \bigcup_{i\in I}{U_i},  \,\,\,  I  \text{ countable } \right\} \]
instead of $\mathcal{H}^1_\infty$, and ask whether $\phi(\mathbb{R})$ has a tangent for the cases of interest $1\leq s \leq 4/3$. Then, the arguments in Remark~\ref{RemarkGeometricDefinition} would also apply here so that we could obtain results in the sense of Definition~\ref{DefFalconer} for every $1 \leq s \leq 4/3$. However, complications arise due to the lack of an analogue to the very useful  Lemma~\ref{LemmaHausdorffContent}. In any case, as we will see, the proof for the non-existence of tangents in the $\mathcal{H}^1_\infty$ setting is based on a parametric approach (see Subsection~\ref{Subsection_ParametricApproach})  independent of the measure chosen and the dimension of the set, which suggests that no tangent should exist in the $\mathcal{H}^s_\infty$ setting. To clarify this and to prove it rigorously is a future work deeply related to the computation of the Hausdorff dimension of $\phi(\mathbb{R})$.
\end{rem}

\subsection{A parametric approach}\label{Subsection_ParametricApproach}

Definition~\ref{DefGeometricTangentForCurves} suits the characteristics of $F=\phi(\mathbb{R})$, but it does not suggest any direct method to work with the parametrisation $\phi$. In the rest of this section, we look for an alternative definition using some parametrisation of the curve and relate it to the geometric approach. 

When looking for a tangent at $\phi(x)$, since $\phi$ is not injective, we propose to work only with points which are close to $\phi(x)$ in parameter; in other words, to consider only $\phi((x-\delta,x+\delta)) $ for some convenient $\delta >0$.
As seen in the introduction, Riemann's function may approach a given point from different directions on the right and on the left (see Figure~\ref{FigureCurve18}). To describe this behaviour, we propose to define tangents on the right, looking only at $\phi((x,x+\delta))$, and on the left, looking only at $\phi((x-\delta,x))$, for some $\delta >0$. Thus, instead of using double cones $S_D$, we use single cones $S(x,\mathbb{V},\varphi)$ consisting on the closure of the set of points $y \in \mathbb{R}^n$ such that $y-x$ forms with $\mathbb{V}$ an angle of at most $\varphi/2$ (see figure~\ref{FigureConeSingle}).
Following the idea of condition~\eqref{TangentCond2Content}, we propose the following definition.
\begin{defin}\label{DefOfTangents}
Let $f : \mathbb{R} \to \mathbb{R}^n$ be a continuous curve and $x \in \mathbb{R}$. We say that $f$ has a \textbf{tangent on the right} at $x$ in direction $\mathbb{V} \in \mathbb{S}^{n-1}$ if
\[ \forall \varphi > 0, \qquad \exists \delta_{\varphi} > 0 \qquad \text{ such that } \qquad f((x,x+\delta_{\varphi})) \subset S(f(x),\mathbb{V},\varphi). \]
We say that $f$ has a \textbf{tangent on the left} at $x$ in direction $\mathbb{V}$ if
\[ \forall \varphi > 0, \qquad \exists \delta_{\varphi} > 0 \qquad \text{ such that } \qquad f((x-\delta_{\varphi},x)) \subset S(f(x),\mathbb{V},\varphi). \]
We say that $f$ has a \textbf{tangent} at $x$ if it has both tangents on the right and on the left and if their directions are diametrically opposite.
\end{defin} 
\begin{rem}\label{RemDependsOnParametrisation}
Using Definition~\ref{DefOfTangents}, we might want to derive tangency properties for a set $F \subset \mathbb{R}^n$ which is parametrised by a continuous function $f:\mathbb{R} \to \mathbb{R}^n$ satisfying $f(\mathbb{R}) = F$. The obvious choice is saying that $F$ has a tangent in $f(x)$ if $f$ has a tangent in $x$. However, this definition depends on the chosen parametrisation $f$ and may yield many undesirable results. For example, if we parametrise the real axis in $\mathbb{R}^2 \simeq \mathbb{C}$, which should have tangent $1 \in \mathbb{S}^1$ everywhere, using the function
\begin{equation}
f(x) = \left\{ \begin{array}{ll}
x, & x<0, \\
0, & 0 < x < 1, \\
x-1, & x \geq 1,
\end{array}
\right.
\end{equation}
then every $\mathbb{V} \in \mathbb{S}^1$ is a tangent of $f$ in $1/2$, which would mean that every direction is a tangent to the real axis at the origin. 
\end{rem}

Remark~\ref{RemDependsOnParametrisation} shows that Definition~\ref{DefOfTangents} is too weak to determine a geometric tangent. Nevertheless, we focus on the reverse direction of the reasoning, since according to the following proposition, Definition~\ref{DefGeometricTangentForCurves} implies Definition~\ref{DefOfTangents}.
\begin{prop}\label{PropConnecting}
Let $f:\mathbb{R} \to \mathbb{R}^n$ be a continuous function, $F = f(\mathbb{R})$, $x \in \mathbb{R}$ and $\mathbb{V} \in \mathbb{S}^{n-1}$. Then, condition \eqref{TangentCond2Content} rewritten as 
\[ \lim_{r \to 0}{ \frac{ \mathcal{H}^1_{\infty} \left( ( F \cap B(f(x),r) ) \setminus S_D(f(x),\mathbb{V},\varphi) \right)  }{r}  } = 0, \qquad \forall \varphi > 0, \]
 implies that
\[ \forall \varphi > 0, \quad  \exists \delta_{\varphi} > 0 \quad : \quad f((x - \delta_{\varphi},x+\delta_{\varphi})) \subset S_D(f(x),\mathbb{V},\varphi). \]
\end{prop}
\begin{proof}
By contradiction, assume there exists $\varphi_0 > 0$ such that for every $\delta >0$ we have 
\[ f((x - \delta,x+\delta)) \not\subset S_D(f(x),\mathbb{V},\varphi_0). \]
Then, there exists a sequence of real numbers $(\delta_n)_{n\in\mathbb{N}}$, which we can assume to be positive (after extracting the subsequence of all positive or all negative terms, and if they are negative the proof is analogous), with $\lim_{n \to \infty}\delta_n = 0$ such that $ f(x + \delta_n) \notin S_D(f(x),\mathbb{V},\varphi_0)$. This property also holds for any cone with an angle $\varphi < \varphi_0$. Define 
\[ \sigma_n = \sup\{ \sigma >0 \quad : \quad f((x+\delta_n - \sigma, x+\delta_n)) \cap S_D(f(x),\mathbb{V},\varphi/2) = \emptyset \}. \] 
By continuity of $f$, $0 < \sigma_n \leq \delta_n$, and also $f(x + \delta_n - \sigma_n) \in \partial S_D(f(x), \mathbb{V}, \varphi/2)$. Call
\[ F_n = f((x+\delta_n-\sigma_n,x+\delta_n)), \]
so that $F_n \cap S_D(f(x),\mathbb{V},\varphi/2) = \emptyset$. Define
\[ r_n = \sup\{ |y - f(x)| \quad : \quad y \in F_n \}, \]
so $\lim_{n\to\infty}r_n = 0$ because $0 < \sigma_n \leq \delta_n \to 0$. Now, since $F_n \subset F$ and $F_n \subset B(f(x),r_n)$, by Lemma~\ref{LemmaHausdorffContent} we may write
\begin{equation*}
\begin{split}
& \mathcal{H}^1_{\infty} \left( ( F \cap B(f(x),r_n) ) \setminus S_D(f(x),\mathbb{V},\varphi/2) \right) \\
& \qquad \geq \mathcal{H}^1_{\infty} \left( \left(  F_n \cap B(f(x),r_n) \right) \setminus S_D(f(x),\mathbb{V},\varphi/2) \right) \\
& \qquad = \mathcal{H}^1_{\infty} \left(  F_n  \right) = \operatorname{diam}F_n.
\end{split}
\end{equation*}
We analyse two cases now.
\begin{itemize}
	\item If $\operatorname{dist}(f(x),f(x+\delta_n)) \leq r_n/2$, then $\operatorname{dist}(f(x+\delta_n),\partial B(f(x),r_n)) \geq r_n/2$. Hence, by the definition of $r_n$, we get $\operatorname{diam}F_n \geq r_n/2$.
	\item Else, if $\operatorname{dist}(f(x),f(x+\delta_n)) > r_n/2$, we have $f(x+\delta_n) \in A(f(x),r_n/2,r_n) \setminus S_D(f(x),\mathbb{V},\varphi)$. Also, since $f(x + \delta_n - \sigma_n) \in \partial S_D(f(x),\mathbb{V},\varphi/2)$, we have 
	\[ \operatorname{diam} F_n  \geq \operatorname{dist} \left( f(x+\delta_n), f(x+\delta_n-\sigma_n) \right) \geq \operatorname{dist} \left( f(x+\delta_n), S_D(f(x),\mathbb{V},\varphi/2) \right)  \]
 and also
	\begin{equation*}
	\begin{split}
	 \operatorname{dist}\left( f(x+\delta_n),S_D(f(x),\mathbb{V},\varphi/2) \right) & \geq \operatorname{dist}\left(  A(f(x),r_n/2,r_n) \setminus S_D(f(x),\mathbb{V},\varphi) ,S_D(f(x),\mathbb{V},\varphi/2) \right) \\
	 & = \frac{r_n}{2}\sin\left( \varphi /4 \right).
	\end{split}
	\end{equation*}	
\end{itemize}
In short, we always get
\[ \operatorname{diam}F_n \geq \frac{r_n}{2} \min{\left(1,\sin{(\varphi/4)} \right)},\]
so
\[  \frac{ \mathcal{H}^1_{\infty} \left( ( F \cap B(f(x),r_n) ) \setminus S_D(f(x),\mathbb{V},\varphi/2) \right) }{r_n}  \geq \frac{\min{\left(1,\sin{(\varphi/4)} \right)}}{2} > 0, \qquad \forall n \in \mathbb{N}, \]
holds for every $0 < \varphi < \varphi_0$. Since $\lim_{n \to \infty}r_n = 0$, the limit in the statement cannot hold.
\end{proof}
\begin{cor}\label{CorMainTool}
Let $f: \mathbb{R} \to \mathbb{R}^n$ be a continuous and non-constant function, $F = f(\mathbb{R})$, $x \in \mathbb{R}$ and $\mathbb{V} \in \mathbb{S}^{n-1}$.  If $\mathbb{V}$ is a tangent of  $F$ at $f(x)$ in the sense of Definition~\ref{DefGeometricTangentForCurves}, then it is a tangent of $f$ at $x$ in the sense of Definition~\ref{DefOfTangents}.

In other words, if $\mathbb{V}$ is not a tangent of $f$ at $x$, then it is not a tangent of $F$ in $f(x)$. 
\end{cor}
\begin{rem}
Condition  \eqref{TangentCond2Content} is independent of $f$, so the conclusion in Proposition~\ref{PropConnecting} holds for \textbf{every} parametrisation $f$ of $F$. Thus, according to Corollary~\ref{CorMainTool}, given a curve-like set $F \subset \mathbb{R}^n$ and $y \in F$, in order to conclude that $\mathbb{V} \in \mathbb{S}^{n-1}$ is not a tangent of $F$ at $y$, it is enough to find an appropriate parametrisation $f$ and $x\in\mathbb{R}$ with $f(x)=y$ not allowing $\mathbb{V}$ as a tangent in $x$. For instance, in the example of Remark~\ref{RemDependsOnParametrisation}, $f$ allowed $i \in \mathbb{S}^1 \subset \mathbb{C}$ as a tangent in $1/2$, but $i$ is not a tangent of the parametrisation $g(x)=x$ at 0. Corollary~\ref{CorMainTool} implies that, as expected, $i$ is not a tangent of the real line at the origin. 
\end{rem}
 
Definition~\ref{DefOfTangents} makes the vector $f(x+r) - f(x)$ converge to the direction $\mathbb{V}$ when $r \to 0$. In the following lemma, we give a more direct way to work with the parametrisation in case $f(x+r) - f(x) \neq 0$.
\begin{lem}\label{DefOfTangentsLimits}
Let $f: \mathbb{R} \to \mathbb{R}^n$ be a continuous curve and $x \in \mathbb{R}$. Then, $f$ has a tangent on the right at $x$ in direction $\mathbb{V} \in \mathbb{S}^{n-1}$ if and only if for every sequence $(r_n)_{n \in \mathbb{N}}$ satisfying $r_n >0$, $f(x+r_n) - f(x) \neq 0$ for all $n \in \mathbb{N}$ and $\lim_{n \to \infty}{r_n} = 0$ we have
\[ \quad \lim_{n \to \infty}{\frac{f(x+r_n) - f(x)}{|f(x+r_n) - f(x)|}} = \mathbb{V}. \]
Moreover, if one of such sequences exists, then the tangent is unique.

The same result applies for tangents on the left by changing $f(x+r_n)$ for $f(x-r_n)$.
\end{lem}
\begin{proof}
The proof consists in writing the cone condition in Definition~\ref{DefOfTangents} in terms of the vector $f(x+r)-f(x)$ for $r>0$. The formulation in terms of sequences comes from the need to take care of the cases $f(x+r) = f(x)$, where the direction of $f(x+r)-f(x) = 0$ is not well-defined.  

On the other hand, if $f$ has tangents $\mathbb{V}_1$ and $\mathbb{V}_2$ in $x$ and if there exists a sequence $(r_n)_n$ as above, then 
\[ \mathbb{V}_1 = \lim_{n \to \infty}{\frac{f(x+r_n) - f(x)}{|f(x+r_n) - f(x)|}}  = \mathbb{V}_2, \]
so the tangent, if it exists, is unique. 
\end{proof}

\begin{rem}
Lemma~\ref{DefOfTangentsLimits} shows that the situation of Remark~\ref{RemDependsOnParametrisation} is the only one in which we may have problems to compute tangents. Indeed, ambiguity will only arise in case the parametrisation is constant in a neighbourhood of $x$.
\end{rem}

\subsection{Main result}

Once the definition of a tangent has been settled, let us specify Theorem~\ref{TheoremIntro}.
\setcounter{section}{1}
\setcounter{thm}{0}
\begin{thm}\label{TheoremIntroSpecified}
Let $\phi$ be Riemann's non-differentiable function \eqref{RiemannFunction},  $t \in \mathbb{R}$ and $\mathbb{V} \in \mathbb{S}^1$. Then, $\mathbb{V}$ is not a tangent of $\phi(\mathbb{R})$ in $\phi(t)$ in the sense of Definition~\ref{DefGeometricTangentForCurves}.
\end{thm}
\setcounter{section}{2}
\setcounter{thm}{13}

It suffices to prove Theorem~\ref{TheoremIntroSpecified} for the set $F = \phi([0,1/2\pi])$. Indeed, $\phi$ has the periodic property
\begin{equation}\label{PeriodicProperty}
\phi\left(t + \frac{1}{2\pi}\right) = \phi(t) + \frac{i}{2\pi}, \qquad \forall t \in \mathbb{R},
\end{equation} 
which can be deduced from \eqref{RelationshipDuistermaatUs} and the fact that $\phi_D$ is periodic of period 2. Consequently, 
\begin{equation}\label{CopiesOfF}
\phi(\mathbb{R}) = \bigcup_{k \in \mathbb{Z}}{\left( F + \frac{ik}{2\pi} \right) },
\end{equation}  
so $\phi(\mathbb{R})$ is a countable union of translations of $F$. Also, by Corollary~\ref{CorMainTool}, it is enough to prove that for any given $t \in [0,1/2\pi]$, no $\mathbb{V} \in \mathbb{S}^1$ is a tangent of $\phi$ in $t$ in the sense of Definition~\ref{DefOfTangents}. Last, to prove this we will use the characterisation given in Lemma~\ref{DefOfTangentsLimits}. 

Let us rescale the variable 
\[ t \in [0,1/2\pi] \mapsto x = 2\pi t \in [0,1] \]
and identify $t = t_x = x/2\pi$ with $x \in [0,1]$. Then,  
\begin{itemize}
	\item If $x = p/q \in [0,1]$ with $p$ and $q$ non-negative and coprime, we call $t_{p/q} = t_{p,q}$ a rational point. 
	\item If $x = \rho \in [0,1]$ is irrational, we call $t_{\rho}$ an irrational point.
\end{itemize}
The main result in this paper is the upcoming Theorem~\ref{TheoremOfTangents}. When $x \in [0,1] \cap \mathbb{Q}$, two different cases were predicted in Figure~\ref{FigureZooms}. In the rationals corresponding to the corner in Figure~\ref{FigureCurve18}, we will see that both the right and left limits in Lemma~\ref{DefOfTangentsLimits} exist, but that they are perpendicular and hence they do not coincide. On the other hand, in the rest of rationals corresponding to the spiral in Figure~\ref{FigureCurve12}, the limit in Lemma~\ref{DefOfTangentsLimits} can take any value in $\mathbb{S}^1$. Finally, if $x \in [0,1]\setminus\mathbb{Q}$, we will show that the limit in Lemma~\ref{DefOfTangentsLimits} can take any value in an open set of $\mathbb{S}^1$. In the three cases, Lemma~\ref{DefOfTangentsLimits} shows that $\phi$ has no tangent in the corresponding point.

\begin{thm}\label{TheoremOfTangents}
Let $\phi$ be Riemann's non-differentiable function \eqref{RiemannFunction} and $x = p/q \in [0,1] \cap \mathbb{Q}$ an irreducible fraction with $q > 0$. 
\begin{enumerate}[(a)]
	\item \label{itm:Corners} If $q \equiv 0,1,3 \pmod{4}$, there exists $e_{p,q} \in \mathbb{S}^1$ an eighth root of unity such that 
	\[ \lim_{r \to 0^+}\frac{\phi(t_{p,q} + r ) - \phi(t_{p,q})}{|\phi(t_{p,q} + r ) - \phi(t_{p,q})|} =  e_{p,q} \frac{1+i}{\sqrt{2}}, \qquad  \lim_{r \to 0^+}\frac{\phi(t_{p,q} - r ) - \phi(t_{p,q})}{|\phi(t_{p,q} - r ) - \phi(t_{p,q})|} =  e_{p,q} \frac{1-i}{\sqrt{2}}. \]
	\item \label{itm:Spirals} If $q \equiv 2 \pmod{4}$, then for any $\mathbb{V} \in \mathbb{S}^1$, there exist sequences $r_n, s_n \to 0^+$ (when $n \to \infty$) such that 
\[ \lim_{n \to \infty}{  \frac{ \phi(t_{p,q} + r_n) - \phi(t_{p,q})}{\left| \phi(t_{p,q} + r_n) - \phi(t_{p,q}) \right| }} = \mathbb{V} = \lim_{n \to \infty}{  \frac{ \phi(t_{p,q} - s_n) - \phi(t_{p,q})}{\left| \phi(t_{p,q} - s_n) - \phi(t_{p,q}) \right| }} \] 
\end{enumerate}
Let $x = \rho \in [0,1] \setminus \mathbb{Q}$. 
\begin{enumerate}[(a)]
	\setcounter{enumi}{2}
	\item \label{itm:Irrationals} There exists an open set $V \subset \mathbb{S}^1$ such that for any $\mathbb{V} \in V$, there exists a sequence $r_n \to 0$ (when $n \to \infty$) such that 
\[  \lim_{n \to \infty}{ \frac{ \phi(t_{\rho} + r_n) - \phi(t_{\rho}) }{ \left| \phi(t_{\rho} + r_n) - \phi(t_{\rho}) \right| } } = \mathbb{V}.  \]
\end{enumerate}
\end{thm}
The cases (\ref{itm:Corners}) and (\ref{itm:Spirals}) will easily follow from the asymptotic behaviour of $\phi$ around rationals that was computed in \cite{Eceizabarrena2019_Part1}. However, the lack of the asymptotic behaviour around irrationals makes the case (\ref{itm:Irrationals}) more complicated. The proof will be based in the continued fraction approximations and the asymptotic behaviour around them.  
We prove each part of Theorem~\ref{TheoremOfTangents} separately in the upcoming Propositions~\ref{Prop013}, \ref{Prop2} and \ref{PropIrrat}.

\subsection{Proof of lemma~\ref{LemmaHausdorffContent}}\label{ProofOfLemmas}
The proof of Lemma~\ref{LemmaHausdorffContent} is based in two auxiliary lemmas
\begin{lem}\label{LemmaDiameter}
Let $A,B \subset \mathbb{R}^n$. If $\overline{A} \cap \overline{B} \neq \emptyset$, then
\[ \operatorname{diam}(A \cup B) \leq \operatorname{diam}A + \operatorname{diam}B. \]
\end{lem}
\begin{proof}
Let $x,y \in A \cup B$. If $x,y \in A$, then $|x-y| \leq \operatorname{diam}A$, and if $x,y \in B$, then $|x-y| \leq \operatorname{diam}B$. In any other case, let $z \in \overline{A} \cap \overline{B}$ so that $ |x-y| \leq |x-z| + |z-y| \leq \operatorname{diam}A + \operatorname{diam}B$.
\end{proof}
\begin{lem}\label{LemmaOrderingCovering}
Let $f:\mathbb{R} \to \mathbb{R}^n$ be a continuous function and $a,b \in \mathbb{R}$ such that $a<b$. Let $N \in \mathbb{N}$ and $U_1,\ldots,U_N \subset \mathbb{R}^n$ a minimal covering by open sets of $f([a,b])$. Then, the sets $U_1,\ldots,U_N$ can be reordered so that 
\[ U_{k+1} \cap \bigcup_{i=1}^{k}U_i \neq \emptyset, \qquad \forall k \in \{1, \ldots, N-1\}. \]
\end{lem}
\begin{proof}
Define $g(x) = f(a + (b-a)x)$ to rescale the interval $[a,b]$ to $[0,1]$. We prove it by induction on $k$. For $k=1$, choose $U_1$ such that $g(0) \in U_1$, and define
\[ \epsilon_1 = \sup\{ \epsilon \in [0,1] \quad : \quad g([0,\epsilon]) \subset U_1 \}. \]
If $\epsilon_1 = 1$, then $N=1$ and we are done. If not, $g(\epsilon_1) \in \partial U_1$. In particular,  $g(\epsilon_1) \notin U_1$ because $U_1$ is open, so $\exists U_2 \neq U_1$ such that $g(\epsilon_1) \in U_2$. Moreover, since $U_2$ is also open, $U_1 \cap U_2 \neq \emptyset$.

Assume now that $U_1,\ldots,U_k$ are ordered according to the statement and define
\[ \epsilon_k = \sup\{ \epsilon \in [0,1] \quad : \quad g([0,\epsilon]) \subset \bigcup_{i=1}^kU_i \}. \]
If $\epsilon_k = 1$, then $N=k$ and we are done. If not, $g(\epsilon_k) \in \partial\left( \bigcup_{i=1}^kU_i \right)$, and since all $U_1,\ldots, U_k$ are open, $g(\epsilon_k) \notin \bigcup_{i=1}^kU_i$. Then, there exists $U_{k+1}$ different from the previous ones such that $g(\epsilon_k) \in U_{k+1}$. Moreover, $U_{k+1}$ is also open, so $U_{k+1} \cap \bigcup_{i=1}^kU_i \neq \emptyset$.
\end{proof}

\begin{proof}[Proof of Lemma~\ref{LemmaHausdorffContent}]
For the first part, $f([a,b]) $ covers itself, so 
\[ \mathcal{H}^1_{\infty}(f([a,b])) \leq \operatorname{diam} f([a,b]) . \]
Let $\{U_i\}_{i\in I}$ be a covering by open sets of $f([a,b])$. By compactness, we may  extract a finite subcovering $\{U_i\}_{i=1}^N$, where $N \in \mathbb{N}$. Then, 
\begin{equation}\label{LowerBoundForH1}
 \operatorname{diam}\left( f([a,b]) \right) \leq \operatorname{diam} \left( \bigcup_{i=1}^NU_i \right) \leq \sum_{i=1}^N{\operatorname{diam}U_i} \leq \sum_{i\in I}{\operatorname{diam}U_i}, 
\end{equation}
where the second inequality holds by Lemma~\ref{LemmaOrderingCovering} and induction on Lemma~\ref{LemmaDiameter}. As a consequence, $  \operatorname{diam} f([a,b])   \leq \mathcal{H}^1_{\infty}(f([a,b]))$.

For the second part, the upper bound is trivial if we choose $B(f(x),r)$ as a covering. Assume there exists $\epsilon>0$  such that $f(x+\epsilon) \notin B(f(x),r)$ and define
\[ \epsilon_0 = \sup\{ \epsilon > 0 \quad : \quad f([x,x+\epsilon]) \subset B(f(x),r)  \}. \]
Then, $f(x+\epsilon_0) \in \partial B(f(x),r)$, so $|f(x+\epsilon_0) - f(x)| = r$. By the first part,
\begin{equation*}
\begin{split}
\mathcal{H}^1_{\infty}(f(\mathbb{R}) \cap B(f(x),r)) & \geq \mathcal{H}^1_{\infty}(f((x,x+\epsilon_0)) \cap B(f(x),r))  = \mathcal{H}^1_{\infty}(f((x,x+\epsilon_0))) \\
& = \operatorname{diam} \left( f((x,x+\epsilon_0)) \right) \\
& \geq |f(x+\epsilon_0) - f(x)| = r.
\end{split}
\end{equation*} 
The case where an $\epsilon>0$ exists such that $f(x-\epsilon) \notin B(f(x),r)$ is analogous.
\end{proof}

\section{Tangents in rationals with $q \equiv 0,1,3 \pmod{4}$}\label{Section013}

In \cite{Eceizabarrena2019_Part1}, the author computed the asymptotic behaviour of $\phi$ around every rational point $t_{p,q}$. Let us briefly explain the main ideas here. The behaviour around 0 can be computed directly using the Poisson summation formula and the asymptotic behaviour of the error function. The one around $t_{1,2}$ can be deduced at once by the identity $\phi(t_{1,2}+h)-\phi(t_{1,2}) = \phi(4h)/2-\phi(h)$. Then, it is proved that the behaviour around every other rational point can be reduced to one of the two, so that the behaviour around $t_{p,q}$ is essentially a rescaling of the behaviour around either 0 or $t_{1,2}$ that depends mainly on $q$. This reduction is performed using the well-known interaction between Jacobi's $\theta$-function, which can be seen as the derivative of $\phi$ (at least formally, see \eqref{DerivationOfPhi}), and the $\theta$-modular group, $\Gamma_\theta$. What lies behind this mathematical procedure is a surprising relationship between Gauss sums, the optical Talbot effect and several symmetries of the Schr\"odinger equation. 

Let us write here the asymptotic behaviour of $\phi$ around points $t_{p,q}$ with $q \equiv 0,1,3 \pmod{4}$ that was given in \cite[Proposition 6.1]{Eceizabarrena2019_Part1}.
\begin{prop}\label{PropositionFromPart1}
If $p,q \in \mathbb{N}$ such that $p<q$, $\operatorname{gcd}(p,q)=1$ and $q \equiv 0,1,3 \pmod{4}$, then
\begin{equation}\label{Asymptotic013}
\begin{split}
\phi(t_{p,q} + h) - \phi(t_{p,q}) & = e_{p,q}\, \frac{1+i}{\sqrt{2\pi}}\,\frac{h^{1/2}}{\tilde{q}^{1/2}} - 4\,e_{p,q}\, \frac{1-i}{\sqrt{2\pi}}\,Y(b(h))\,\tilde{q}^{3/2}\,h^{3/2} + O\left( q^{7/2}h^{5/2} \right),
\end{split}
\end{equation} 
when $|h| < \left( 4\pi \frac{|c_{\pm}|}{\tilde{q}}\,\tilde{q}^2 \right)^{-1}$. Here, there exist $c_+,c_- \in \mathbb{R}$ such that $\tilde{q} < c_+ < 4\tilde{q}$ and $-4\tilde{q} < c_- < -\tilde{q}$, define $c_\pm = c_+$ when  $h>0$ and $c_\pm = c_-$ when $h<0$, and 
\begin{equation}\label{DefOfB}
\tilde{q} =\left\{
 \begin{array}{ll}
q, & \text{if } q \text{ is odd,} \\
q/2, & \text{if } q \text{ is even,}
\end{array}
\right. ,
\qquad   
Y(h) = \sum_{k=1}^{\infty}{\frac{e^{ik^2/(4h)}}{k^2}}, 
\qquad 
b(h) = \left\{ \begin{array}{ll}
\frac{\tilde{q}^2h}{1 + 4\pi c_+ \tilde{q} h}, & \text{ when } h  \geq 0, \\
\frac{\tilde{q}^2h}{1 + 4\pi c_- \tilde{q} h}, & \text{ when } h  < 0.
\end{array}\right. 
\end{equation}
Also,  $\sqrt{-1} = -i$ when $h <0$ and $e_{p,q}$ is an eighth root of unity only depending on $p$ and $q$.

\end{prop}

A direct consequence of this asymptotic behaviour is
\[ \lim_{h \to 0^+}{ \frac{ \phi(t_{p,q} + h ) - \phi(t_{p,q}) }{ e_{p,q}\, \frac{1+i}{\sqrt{2\pi}}\,\frac{h^{1/2}}{\tilde{q}^{1/2}}  } } = 1. \]
The term in the denominator is the parametrisation of a straight line with direction $e_{p,q}\,(1+i)/\sqrt{2}$, so we expect that $F$ has a tangent to the right of $t_{p,q}$ given by $e_{p,q}\, (1+i)/\sqrt{2}$.

\begin{prop}\label{Prop013}
Let $p,q \in \mathbb{N}$ be such that $q >0$, $\operatorname{gcd}(p,q) = 1$ and  $q \equiv 0,1,3 \pmod{4}$. Then, 
\[ \lim_{h \to 0^+}\frac{\phi(t_{p,q} + h ) - \phi(t_{p,q})}{|\phi(t_{p,q} + h) - \phi(t_{p,q})|} =  e_{p,q} \frac{1+i}{\sqrt{2}}, \]
so $\phi$ has a tangent on the right at $t_{p,q}$ in direction $e_{p,q}\,(1+ i)/\sqrt{2}$.  Also,
\[ \lim_{h \to 0^+}\frac{\phi(t_{p,q} - h ) - \phi(t_{p,q})}{|\phi(t_{p,q} - h ) - \phi(t_{p,q})|} =  e_{p,q} \frac{1-i}{\sqrt{2}}, \]
so it has a tangent on the left in direction $e_{p,q}\,(1- i)/\sqrt{2}$. 
\end{prop}

\begin{proof}
It is enough to work with the shortened version of the asymptotic \eqref{Asymptotic013},
\[ \phi(t_{p,q} + h ) - \phi(t_{p,q})  = e_{p,q}\, \frac{1+i}{\sqrt{2\pi}}\,\frac{h^{1/2}}{\tilde{q}^{1/2}} + O\left( q^{3/2}h^{3/2} \right). \]
Let $h >0$. Then,
\[ \frac{\phi(t_{p,q} + h ) - \phi(t_{p,q})}{|\phi(t_{p,q} + h ) - \phi(t_{p,q})|} = \frac{ e_{p,q} \frac{1+i}{\sqrt{2\pi}} \frac{h^{1/2}}{\tilde{q}^{1/2}} \left( 1 + O(q^2h) \right) }{  \frac{1}{\sqrt{\pi}} \frac{h^{1/2}}{\tilde{q}^{1/2}} \left| 1 + O(q^2h) \right| } =  e_{p,q} \frac{1+i}{\sqrt{2}} \frac{ 1 + O(q^2h) }{ \left| 1 + O(q^2h) \right| }, \]
so 
\[ \lim_{h \to 0^+}\frac{\phi(t_{p,q} + h ) - \phi(t_{p,q})}{|\phi(t_{p,q} + h ) - \phi(t_{p,q})|} =  e_{p,q} \frac{1+i}{\sqrt{2}}. \]
On the other hand, with the branch $\sqrt{-1} = -i$,
\[ \phi(t_{p,q} - h ) - \phi(t_{p,q})  = e_{p,q}\, \frac{1-i}{\sqrt{2\pi}}\,\frac{h^{1/2}}{\tilde{q}^{1/2}} + O\left( q^{3/2}h^{3/2} \right). \]
The same procedure as above shows that 
\[ \lim_{t \to 0}\frac{\phi(t_{p,q} - h ) - \phi(t_{p,q})}{|\phi(t_{p,q} - h ) - \phi(t_{p,q})|} =  e_{p,q} \frac{1-i}{\sqrt{2}}. \] 
\end{proof}

\section{Tangents in rationals with $q \equiv 2\pmod{4}$}\label{Section2}
The asymptotic behaviour of $\phi$ around a rational $t_{p,q}$ with $q \equiv 2\pmod{4}$ is considerably different, as was proved in \cite[Proposition 7.1]{Eceizabarrena2019_Part1}. We gather it in the following proposition.
\begin{prop}
If $p,q \in \mathbb{N}$ such that $p<q$, $\operatorname{gcd}(p,q)=1$ and $q \equiv 2 \pmod{4}$, then 
\begin{equation}\label{Asymptotic2}
\phi(t_{p,q} + h) - \phi(t_{p,q}) = -16\,e_{p,q}\,\frac{1-i}{\sqrt{2\pi}}\, Z(b(h))\,\tilde{q}^{3/2}h^{3/2} + O(q^{7/2}h^{5/2}), 
\end{equation} 
where 
\[ Z(h) = \sum_{\substack{k = 1\\ k \text{ odd}}}^{\infty}{ k^{-2}\,e^{-ik^2/(16h)} } \]
and all the rest of definitions and the range of validity for $h$ are the same as in Proposition~\ref{PropositionFromPart1}.
\end{prop}
We remark that the principal term is of the order of $h^{3/2}$, and not $h^{1/2}$ as in the previous case. The reason for which $\phi$ does not have a tangent in these points is that it follows a spiralling pattern generated by $Z$. Therefore, there will be parts of the curve in every direction arbitrarily close to the point.

\begin{prop}\label{Prop2}
Let $p, q \in \mathbb{N}$ be such that $q>0$, $\operatorname{gcd}(p,q) = 1$ and  $q \equiv 2 \pmod{4} $. For any $\mathbb{V} \in \mathbb{S}^1$, there exist sequences $r_n,s_n \to 0^+$ (when $n \to \infty$) such that 
\[ \lim_{n \to \infty}{  \frac{ \phi(t_{p,q} + r_n) - \phi(t_{p,q})}{\left| \phi(t_{p,q} + r_n) - \phi(t_{p,q}) \right| }} = \mathbb{V} = \lim_{n \to \infty}{  \frac{ \phi(t_{p,q} - s_n) - \phi(t_{p,q})}{\left| \phi(t_{p,q} - s_n) - \phi(t_{p,q}) \right| }} \] 
\end{prop}
\begin{proof}
Let $h >0$, so 
\begin{equation*}
\begin{split}
\frac{\phi(t_{p,q} + h) - \phi(t_{p,q})}{|\phi(t_{p,q} + h) - \phi(t_{p,q})|} & =  \frac{ -16e_{p,q}\frac{1-i}{\sqrt{2\pi}}\tilde{q}^{3/2}h^{3/2} \left( Z(b(h)) + O(q^2h) \right) }{ -16\frac{1}{\sqrt{\pi}}\tilde{q}^{3/2}h^{3/2} \left| Z(b(h)) + O(q^2h) \right| } \\
& = e_{p,q}\frac{1-i}{\sqrt{2}} \frac{   Z(b(h)) + O(q^2h)   }{ \left|Z(b(h)) + O(q^2h) \right| }.
\end{split}
\end{equation*}  
To take the limit when $h \to 0$, we need to understand  $\lim_{h \to 0}{  Z(b(h)) }$,
which, by continuity of $Z$ and since $b(h) \sim \tilde{q}^2h$ are equivalent infinitesimals, can be written as
\begin{equation*}
\lim_{h \to 0}{ Z(b(h)) } = \lim_{h \to 0}{ Z (\tilde{q}^2 h) } = \lim_{t \to +\infty}{ Z \left( \tilde{q}^2 t^{-1} \right) }.
\end{equation*} 
The function $Z(\tilde{q}^2/\cdot)$ has period $32\pi\tilde{q}^2$. If $B(x,r)$ denotes the ball centered at $x$ and with radius $r$, then
\[ Z(\tilde{q}^2t^{-1}) =   e^{-i\frac{t}{16\tilde{q}^2} } + \sum_{\substack{k = 3 \\ k \text{ odd}}}^{\infty}{ \frac{e^{-i\frac{k^2}{16\tilde{q}^2} t} }{k^2} } \in B\left(e^{-\frac{it}{16\tilde{q}^2} }, 1/4 \right) \]
because $\sum_{\substack{k = 3 \\ k \text{ odd}}}^{\infty}{k^{-2} } = \pi^2/8 - 1\leq 1/4$, so
\[ \operatorname{arg}{ Z(\tilde{q}^2t^{-1})} = -\frac{t}{16\tilde{q}^2} + \eta(t) \qquad \text{ such that } \qquad |\eta(t)| \leq \frac{\pi}{4}. \]
From the fact that $\operatorname{arg}Z(\tilde{q}^2/(-48\pi\tilde{q}^2)) \geq 2\pi$ and $\operatorname{arg}Z(\tilde{q}^2/(16\pi\tilde{q}^2)) \leq 0$, by continuity and periodicity, the function $\operatorname{arg}Z(\tilde{q}^2 t^{-1})$ takes all possible values when $t \in [0,32\pi\tilde{q}^2]$. In other words, for any $\zeta \in [0,2\pi]$, there exists $t_{\zeta} \in [0,32\pi\tilde{q}^2]$ such that $\operatorname{arg}Z(\tilde{q}^2t_{\zeta}^{-1}) = \zeta$. Equivalently, for any $\mathbb{V} \in \mathbb{S}^1$, there exists $t_{\mathbb{V}} \in [0,32\pi\tilde{q}^2]$ such that $Z(\tilde{q}^2t_{\mathbb{V}}^{-1}) / |Z(\tilde{q}^2t_{\mathbb{V}}^{-1})| = \mathbb{V}$. Choose $\tau_n = 32\pi\tilde{q}^2n+t_{\mathbb{V}}$ so that $Z(\tilde{q}^2\tau_n^{-1}) / |Z(\tilde{q}^2\tau_n^{-1})| = \mathbb{V}$ for all $n \in \mathbb{N}$. This way,
\[ \lim_{n \to \infty} \frac{\phi(t_{p,q} + \tau_n^{-1}) - \phi(t_{p,q})}{|\phi(t_{p,q} + \tau_n^{-1}) - \phi(t_{p,q})|}  = e_{p,q} \frac{1-i}{\sqrt{2}} \lim_{n \to \infty}{ \frac{Z(\tilde{q}^2\tau_n^{-1})}{|Z(\tilde{q}^2\tau_n^{-1})|} } = e_{p,q} \frac{1-i}{\sqrt{2}}\mathbb{V}. \]
In the same way, choosing $\sigma_n = 32\pi\tilde{q}^2n - t_{\mathbb{V}}$ for every $n \in \mathbb{N}$, recalling that  $\sqrt{-1} = -i$, we get
\[ \lim_{n \to \infty} \frac{\phi(t_{p,q} - \sigma_n^{-1}) - \phi(t_{p,q})}{|\phi(t_{p,q} - \sigma_n^{-1}) - \phi(t_{p,q})|} =  e_{p,q} \frac{1+i}{\sqrt{2}} \lim_{n \to \infty}{ \frac{Z(-\tilde{q}^2\sigma_n^{-1})}{|Z(-\tilde{q}^2\sigma_n^{-1})|} } =  e_{p,q} \frac{1+i}{\sqrt{2}} \frac{Z(\tilde{q}^2t_{\mathbb{V}}^{-1})}{|Z(\tilde{q}^2t_{\mathbb{V}}^{-1})|} = e_{p,q} \frac{1+i}{\sqrt{2}}\,\mathbb{V}.   \]
\end{proof}

\section{Tangents in irrationals}\label{SectionIrrationals}

\begin{prop}\label{PropIrrat}
Let $\rho \in [0,1] \setminus \mathbb{Q}$. There exists an open set $V \subset \mathbb{S}^1$ such that for any $\mathbb{V} \in V$, there exists a sequence $r_n \to 0$ with $n \to \infty$ such that 
\[  \lim_{n \to \infty}{ \frac{ \phi(t_{\rho} + r_n) - \phi(t_{\rho}) }{ \left| \phi(t_{\rho} + r_n) - \phi(t_{\rho}) \right| } } = \mathbb{V}.  \]
\end{prop}

In the case of an irrational $\rho \in [0,1]$, we have no asymptotics to describe the behaviour of the curve around it. To overcome this difficulty, we will find convenient rational approximations $p/q$ and use the asymptotic behaviour $\phi(t_{p,q} + h) - \phi(t_{p,q})$  \eqref{Asymptotic013} choosing $h = t_{\rho} - t_{p,q} = (\rho - p/q)/2\pi$. However, recall that the asymptotic behaviour is precise only when $h \to 0$.

Let $\rho_n = p_n/q_n$, $n \in \mathbb{N}$ be the approximations given by the continued fraction of $\rho$, so that $ | \rho - \rho_n | <  q_n^{-2}$ (see, for instance, \cite{Khinchin1964}). In order to work with a more precise measurement of this error, define $K_n = K_n(\rho)$  by
\[   | \rho - \rho_n | = \frac{K_n}{q_n^2}\,, \qquad \qquad 0 < K_n < 1, \qquad \text{ for all } n \in \mathbb{N}. \]
These coefficients are in the correct scale in the asymptotic at $t_{p,q}$. Indeed, following \cite[Proposition 6.1]{Eceizabarrena2019_Part1}, if $q \equiv 0,1,3 \pmod{4}$, or rescaling \eqref{Asymptotic013} with $h = s/\tilde{q}^2$, the asymptotic becomes
\begin{equation}\label{Asymptotic013Rescaled}
\begin{split}
\phi(t_{p,q} + s/\tilde{q}^2 ) - \phi(t_{p,q}) & = e_{p,q}\, \frac{1+i}{\sqrt{2\pi}}\,\frac{s^{1/2}}{\tilde{q}^{3/2}} - 4e_{p,q} \frac{1-i}{\sqrt{2\pi}}\,Y(b(s/\tilde{q}^2))\, \frac{s^{3/2}}{\tilde{q}^{3/2}} + O\left( q^{-3/2}s^{5/2} \right) \\
& = \frac{e_{p,q}}{\tilde{q}^{3/2}}\,\left( \frac{1+i}{\sqrt{2\pi}}s^{1/2} - 4 \frac{1-i}{\sqrt{2\pi}}\,Y(b(s/\tilde{q}^2))\,s^{3/2} + O\left( s^{5/2} \right) \right).
\end{split}
\end{equation} 
Since $n \to \infty$, $q$ is no longer fixed and it has to be taken care of. According to the definition \eqref{DefOfB}, $b = b_q$ depends on $q$, and 
\[  b_q\left( s/\tilde{q}^2 \right) = \beta_q(s) =  \frac{s}{1 + 4\pi \frac{c_\pm}{\tilde{q}}s}. \]
The above depends only on $ c_\pm/\tilde{q}$, bounded by $1 \leq |c_\pm/\tilde{q}| \leq 4$. Also up to rescaling the variable and the image, the resulting asymptotic is very similar to that in $0$ (see \cite[Proposition 4.1]{Eceizabarrena2019_Part1}, or \eqref{Asymptotic013} with $q=1$). 
In other words, $\phi$ does around $t_{p,q}$ essentially the same as around 0 at a smaller scale.  

\begin{rem}
If the rescaled asymptotic \eqref{Asymptotic013Rescaled}
is centered in the approximations $p_n/q_n$, 
for simplicity we may write
\begin{equation}\label{DefOfBetan}
 \beta_n(s) = \frac{s}{1 + 4\pi \frac{c_n}{\tilde{q}_n}s} 
\end{equation}
instead of $\beta_{q_n}(s)$ to express that the dependence is now on $n\to\infty$.
\end{rem}

\begin{rem}\label{RemarkOddConvergents}
Infinitely many approximations $p_n/q_n$ coming from the continued fraction of $\rho$ are such that $q_n$ is odd. This is easy to check from the recurrence relation
\[ q_{n+2} = a_{n+2}q_{n+1} + q_n, \qquad \forall n \in \mathbb{N}, \]
where $\rho = [a_0; a_1,a_2,a_3,\ldots]$ is the continued fraction (see \cite{Khinchin1964}).  In this case, $\tilde{q} = q$, which simplifies the notation and allows to work only with \eqref{Asymptotic013Rescaled}. Hence, from now on, we work with the subsequence of these approximations, which for simplicity we rename again simply as $p_n/q_n$.
\end{rem} 

According to the above, the position of $t_\rho$ in the rescaled asymptotic \eqref{Asymptotic013Rescaled} depends on the sequence $(K_n)_{n\in\mathbb{N}}$. 
However, it is easy to build examples in which  $\lim_{n \to \infty}K_n$ does not exist. In view of this, after proceeding as in Remark~\ref{RemarkOddConvergents}, call 
\[ K = \liminf_{n \to \infty}{K_n}, \]
and extract a second subsequence such that, after renaming it again as the original sequence, $\lim_{n \to \infty}K_n = K$. Then, our approach depends very much on the value of $K \in [0,1]$.
\begin{itemize}
	\item If $K>0$, we first take a third subsequence to manage the effect of $c_n/q_n$ so that the asymptotics in \eqref{Asymptotic013Rescaled}, which depend on $c_n/q_n$, tend to some \textit{limit asymptotic}. The irrational point tends to stabilise somewhere far from the origin in this limit asymptotic. In this stable setting, we will be able to conclude.  
	\item If $K=0$, the irrational point tends to the origin in the rescaled asymptotic \eqref{Asymptotic013Rescaled}. We need to rescale a second time to exploit the selfsimilar properties of $F$ described in \cite[Proposition 6.1]{Eceizabarrena2019_Part1}, which show that the better copies of itself are generated the closer we get to the origin. Once we detect the copy in which $t_{\rho}$ is, we will be able to conclude and deduce that no tangent can exist.
\end{itemize}

Let us rewrite the asymptotic \eqref{Asymptotic013Rescaled} at each of the approximations $p_n/q_n$ evaluated in $K_n/2\pi$. The main parameter is now $n \in \mathbb{N}$, so  writing $e_{p_n,q_n} = e_n$, we have
\begin{equation}\label{Asymptotic013RescaledApproximations}
\phi(t_{\rho} ) - \phi(t_{p_n,q_n}) =   \frac{1+i}{\sqrt{2\pi}}\, \frac{e_n}{ q_n^{3/2} }\left(  \sqrt{\frac{K_n}{2\pi}} + 4i\, Y(\beta_n\left(K_n/2\pi\right))\left( \frac{K_n}{2\pi} \right)^{3/2} + O(K_n^{5/2})    \right).
\end{equation}
In the case $K=0$, we recover from \cite[Proposition 6.1]{Eceizabarrena2019_Part1} the self-similar version of \eqref{Asymptotic013RescaledApproximations}, which is
\begin{equation}\label{Asymptotic013Selfsimilar}
\phi\left(t_{\rho} \right) - \phi(t_{p_n,q_n}) = \frac32 \frac{1+i}{\sqrt{2\pi}} \frac{e_n}{q_n^{3/2}} \left[   \sqrt{\frac{K_n}{2\pi}} + \frac{8\pi^2}{3}i\left( \frac16 - \frac{i}{2\pi}\frac{c_n}{q_n} - 2\phi\left( \frac{-1}{16\pi^2\beta_n(\frac{K_n}{2\pi})}\right)\right)\left( \frac{K_n}{2\pi} \right)^{\frac32}   + O\left( K_n^{\frac52} \right)    \right]
\end{equation}
for every $n \in \mathbb{N}$. 
In both cases \eqref{Asymptotic013RescaledApproximations} and \eqref{Asymptotic013Selfsimilar}, we name the rescaled asymptotic as 
\begin{equation}\label{DefOfHn}
\begin{split}
H_n(s) & =  \sqrt{s} + 4i \, Y(\beta_n(s)) \,s^{3/2} + O(s^{5/2})  \\
& = \frac32 \left( \sqrt{s} + \frac{8\pi^2}{3}i\left( \frac16 - \frac{i}{2\pi}\frac{c_n}{q_n} - 2\phi\left( \frac{-1}{16\pi^2\beta_n(s)} \right) \right)s^{3/2} + O\left( s^{5/2} \right) \right).
\end{split}
\end{equation}
so that for every $n \in \mathbb{N}$, both can be read as
\begin{equation}\label{PhiInTermsOfH}
\phi\left(t_{\rho} \right) - \phi(t_{p_n,q_n}) =  \frac{1+i}{\sqrt{2\pi}} \frac{e_n}{q_n^{3/2}} H_n\left(\frac{K_n}{2\pi}\right). 
\end{equation} 

\begin{rem}
The approximations by continued fractions approach $\rho$ alternately from the right and from the left, but after having extracted subsequences, this may no longer be true. However, either the approximations from the right or those from the left must be infinitely many. We extract this infinite subsequence so that every approximation is on the same side of $\rho$. Moreover, we may assume they are on the left, so $\rho - \rho_n >0$, and therefore work with $s>0$ and with $c_n = (c_\pm)_n = (c_+)_n>0$ in \eqref{DefOfHn}. Indeed, if this was not the case, the asymptotics \eqref{DefOfHn} are also valid with $s = -|s| < 0$ recalling that $\sqrt{-1}=-i$, and the proof is analogous. 

Before splitting the analysis into the cases $K=0$ and $K>0$, let us sum up the subsequences we have extracted: 
\begin{itemize}
	\item $q_n$ is odd, so $\tilde{q}_n = q_n$ and the corresponding asymptotic behaviour is \eqref{PhiInTermsOfH}.
	\item $\lim_{n \to \infty}{ K_n } = K \in [0,1],$
	\item All $p_n/q_n$ approach $\rho$ from the left, so we work with $s>0$ in \eqref{DefOfHn}.
\end{itemize}
\end{rem}

%
%

\subsection{$\boldsymbol{K>0}$}

In the setting of the asymptotic \eqref{DefOfHn}, the position of $t_{\rho}$ tends to stabilise somewhere far from the origin. However, there are two drawbacks to be solved:
\begin{enumerate}[(i)]
	\item The asymptotic \eqref{DefOfHn} depends on $n\in\mathbb{N}$.
	\item Since $s = K_n/2\pi \to K/2\pi >0$, we lack control of the error term in \eqref{DefOfHn}.
\end{enumerate}

To solve the second issue, we recover a closed expression for \eqref{DefOfHn} from \cite[Section 6]{Eceizabarrena2019_Part1},
\begin{equation}\label{DefOfHnClosed}
H_n(s) = \sqrt{\pi}\frac{1-i}{\sqrt{2}} \left( \frac{\phi(\beta_n(s))}{(1-4\pi\frac{c_n}{q_n}\beta_n(s))^{3/2}} - 6\pi\frac{c_n}{q_n}\int_{0}^{\beta_n(s)}{\frac{\phi(r)}{(1 - 4\pi \frac{c_n}{q_n}r)^{5/2}}\,dr  } \right),  
\end{equation} 
where there is no longer an error term. On the other hand, regarding issue (i), let 
\begin{equation}\label{DefOfA}
a  =\liminf_{n \to \infty}{ \frac{c_n}{q_n}} \in [1,4]
\end{equation} 
 and take a subsequence of the approximations which, after relabelling, satisfies $\lim_{n \to \infty}{c_n/q_n} = a$. Since \eqref{DefOfHnClosed} depends only on $c_n/q_n$ (because so does $\beta_n$), we expect it to  converge to 
\begin{equation}\label{DefOfH}
\begin{split}
H(s) & = \sqrt{\pi}\frac{1-i}{\sqrt{2}} \left( \frac{\phi(\beta(s))}{(1-4\pi a \beta(s))^{3/2}} - 6\pi a \int_{0}^{\beta(s)}{\frac{\phi(r)}{(1 - 4\pi a r)^{5/2}}\,dr  } \right) \\
& = \sqrt{s} + 4i\, Y(\beta\left(s\right)) \, s ^{3/2} + O(s^{5/2}),
\end{split}
\end{equation} 
where
\begin{equation}\label{DefOfBeta}
\beta(s) = \lim_{n \to \infty}\beta_n(s) = \frac{s}{1 + 4\pi a s}. 
\end{equation}
Indeed, that is precisely what happens.
\begin{lem}\label{LemmaUniformConvergence}
Let $M >0$. The sequence of functions $H_n$ converges uniformly to $H$ in $[0,M]$. In other words,
\[ \lim_{n \to \infty}{\lVert H_n - H \rVert_{L^{\infty}([0,M])}} = 0. \]
\end{lem}
\begin{proof}
See Appendix~\ref{Appendix}.
\end{proof}

\begin{figure}[h]
\includegraphics[scale=0.82]{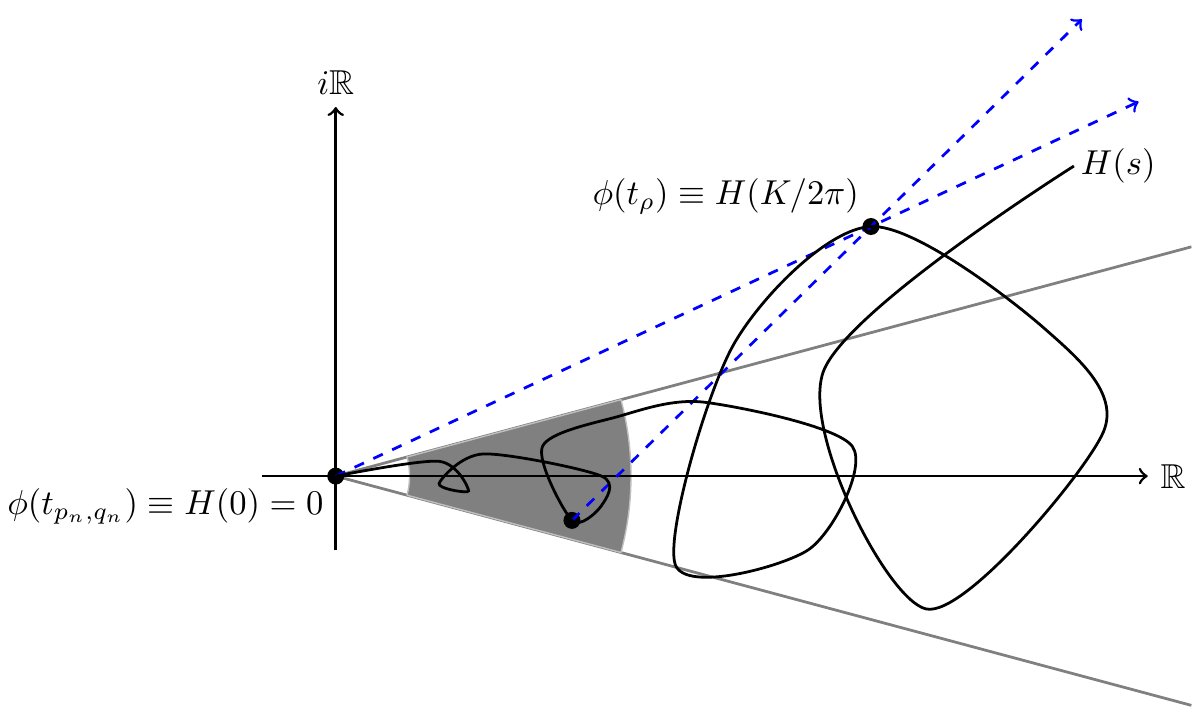}
\caption{Schematic representation of the curve $H(s)$. According to Lemma~\ref{LemmaUniformConvergence}, the $q_n^{3/2}$-rescaled neighbourhood of $\phi(t_{\rho})$ converges to this situation when $n \to \infty$. As approximations to $\phi(t_\rho)$, we use $\phi(t_{p_n,q_n})$ (the origin in the picture), but we might also use any other point between $\phi(t_{p_n,q_n})$ and $\phi(t_{\rho})$, for instance one lying in the shaded region. Each approximation leads to different tangents, dashed in blue, and consequently, no tangent to $\phi(t_\rho)$ can exist.}
\label{GeneralSituation}
\end{figure}

To prove Proposition~\ref{PropIrrat}, we use the characterisation with limits in Lemma~\ref{DefOfTangentsLimits} with the approximations $p_n/q_n$. The idea is that, apart from $\phi(t_{p_n,q_n})$, we use alternative approximations lying between $t_{\rho}$ and $t_{p_n,q_n}$, each of them leading to a different tangent (see Figure~\ref{GeneralSituation}). 

Following Lemma~\ref{DefOfTangentsLimits} and \eqref{PhiInTermsOfH}, we compute
\begin{equation}\label{ConditionTangent}
\lim_{n \to \infty}{ \frac{ \phi(t_{p_n,q_n}) - \phi(t_{\rho}) }{|\phi(t_{p_n,q_n}) - \phi(t_{\rho})|} } =   - \frac{1+i}{\sqrt{2}} \lim_{n \to \infty}{ e_n \, \frac{  H_n(K_n/2\pi)}{\left|   H_n(K_n/2\pi) \right|} } .
\end{equation} 
From Lemma~\ref{LemmaUniformConvergence} and $\lim_{n\to\infty}{K_n} = K > 0$, we deduce that $\lim_{n\to\infty}{H_n(K_n/2\pi)} = H(K/2\pi)$. Hence, we need different approaches if $H(K/2\pi) \neq 0$ or if $H(K/2\pi) = 0$, which might happen.

\subsubsection{$\boldsymbol{H(K/2\pi) \neq 0}$}
We write the limit \eqref{ConditionTangent} as
\[  - \frac{1+i}{\sqrt{2}}\, \frac{  H(K/2\pi)}{\left|   H(K/2\pi) \right|} \lim_{n \to \infty}{ e_n  }. \]
This limit may or may not exist, but since we want to show that the curve approaches $\phi(t_{\rho})$ from any direction in an open set of $\mathbb{S}^1$, taking into account that $e_n^8=1$ for all $n \in \mathbb{N}$, we take a subsequence of $p_n/q_n$ such that, after relabelling, $e_n = e \in \mathbb{C}$ is constant. Thus, the limit and hence the candidate to be the tangent is
\begin{equation}\label{TangentFromApproximation1}
- \frac{1+i}{\sqrt{2}}\,e \, \frac{  H(K/2\pi)}{\left|   H(K/2\pi) \right|} . 
\end{equation}
Let us work with other approximations by letting $Q \in \mathbb{R}$ be such that 
\[ 0 < Q < \frac{K}{2} < K_n \]
for big enough $n \in \mathbb{N}$. Then, the tangent can also be computed with the approximations
$\phi(t_{p_n,q_n} + Q/(2\pi q_n^2))$, which by the choice of $Q$ are always between $\phi(t_{p_n,q_n})$ and $\phi(t_{\rho})$ in parameter. Consequently, by \eqref{PhiInTermsOfH}, the limit for the tangent as in Lemma~\ref{DefOfTangentsLimits} is
\begin{equation}\label{TangentFromApproximation2}
\begin{split}
& \lim_{n \to \infty}{ \frac{ \phi(t_{p_n,q_n}+\frac{Q}{2\pi q_n^2}) - \phi(t_{\rho}) }{|\phi(t_{p_n,q_n} + \frac{Q}{2\pi q_n^2}) - \phi(t_{\rho})|} }  \\
& \qquad =  \lim_{n \to \infty}{ \frac{ \phi(t_{p_n,q_n}+\frac{Q}{2\pi q_n^2}) - \phi(t_{p_n,q_n}) + \phi(t_{p_n,q_n}) -  \phi(t_{\rho}) }{ \left|\phi(t_{p_n,q_n} + \frac{Q}{2\pi q_n^2}) - \phi(t_{p_n,q_n}) + \phi(t_{p_n,q_n}) - \phi(t_{\rho}) \right|} }  \\
& \qquad  =  \lim_{n \to \infty}{ \frac{ \frac{1+i}{\sqrt{2\pi}} \frac{e}{q_n^{3/2}} H_n\left(\frac{Q}{2\pi}\right)  -  \frac{1+i}{\sqrt{2\pi}} \frac{e}{q_n^{3/2}} H_n\left(\frac{K_n}{2\pi}\right) }{ \left| \frac{1+i}{\sqrt{2\pi}} \frac{e}{q_n^{3/2}} H_n\left(\frac{Q}{2\pi}\right)  -  \frac{1+i}{\sqrt{2\pi}} \frac{e}{q_n^{3/2}} H_n\left(\frac{K_n}{2\pi}\right) \right| } } \\
& \qquad = \frac{1+i}{\sqrt{2}} \, e  \, \frac{  H\left(Q/2\pi\right)  -   H\left(K/2\pi\right) }{ \left|   H\left(Q/2\pi\right)  -   H\left(K/2\pi\right) \right| } .
\end{split}
\end{equation}  
Since $H(K/2\pi) \neq 0$ and $\lim_{s \to 0}H(s) = 0$, by continuity of $H$ we can choose $0 < Q < K/2$ such that $H(K/2\pi) \neq H(Q/2\pi) \neq 0$. According to Lemma~\ref{DefOfTangentsLimits}, a tangent will definitely not exist if 
\begin{equation}\label{ConditionFromTwoApproximations}
 \frac{  H(K/2\pi)}{\left|   H(K/2\pi) \right|} \neq  \frac{  H\left(K/2\pi\right)  -   H\left(Q/2\pi\right) }{ \left|   H\left(K/2\pi\right)  -   H\left(Q/2\pi\right) \right| } ,
\end{equation} 
and in particular if $H(Q/2\pi) \notin H(K/2\pi)\mathbb{R}$. 

With the definition of $Y$ \eqref{DefOfB} and the asymptotic \eqref{DefOfH} in mind, define the sequences $s_m$ and $\tilde{s}_m$ for $m \in \mathbb{N}$ as
\begin{equation}\label{DefPointsForGoodSum}
\frac{1}{4\beta(s)} = 2\pi m  \Leftrightarrow  s = s_m = \frac{1}{2\pi}\frac{1}{4m - 2a}, \quad \text{ and } \quad \frac{1}{4\beta(s)} = (2 m + 1 )\pi  \Leftrightarrow  s = \tilde{s}_m = \frac{1}{4\pi}\frac{1}{2m + 1 - a} .  
\end{equation}  
Then, 
\[ Y(\beta(s_m)) = \sum_{k=1}^\infty \frac{e^{2\pi i m k^2}}{k^2} = \sum_{k=1}^\infty \frac{1}{k^2} = \frac{\pi^2}{6} > 0 \]
and
\[ Y(\beta(\tilde{s}_m)) = \sum_{k=1}^\infty \frac{e^{(2m+1)\pi i k^2}}{k^2} = \sum_{k=1}^\infty \frac{e^{i\pi k^2}}{k^2} = \sum_{k=1}^\infty \frac{(-1)^k}{k^2} = -\frac{\pi^2}{12} < 0. \]
Since  $\lim_{m \to \infty}s_m = 0 = \lim_{m\to\infty}\tilde{s}_m$, the error terms $O(s_m^{5/2})$ and $O(\tilde{s}_m^{5/2})$  in \eqref{DefOfH} are negligible when $m \to \infty$. Also, the real part of $\sqrt{s_m} + 4iY(\beta(s_m))\,s_m^{3/2} $ is $\sqrt{s_m}>0$, and the imaginary part is $4Y(\beta(s_m))\,s_m^{3/2}>0$, so for big enough $m \in \mathbb{N}$, $H(s_m)$ is in the first quadrant of the complex plane. On the other hand, the real part of $\sqrt{\tilde{s}_m} + 4iY(\beta(\tilde{s}_m))\,\tilde{s}_m^{3/2}$ is $\sqrt{\tilde{s}_m}>0$ and the imaginary part is $4Y(\beta(\tilde{s}_m))\,\tilde{s}_m^{3/2}<0$, so $H(\tilde{s}_m)$ is in the fourth quadrant of $\mathbb{C}$ for big enough $m \in \mathbb{N}$. 

Therefore, if $H(K/2\pi)$ is either in the first or the third quadrant, we can choose $Q = 2\pi\tilde{s}_m$ for a big enough $m \in \mathbb{N}$ so that $H(Q/2\pi)$ is in the fourth quadrant, and thus $H(Q/2\pi) \notin H(K/2\pi)\mathbb{R}$. On the other hand, in case $H(K/2\pi)$ is in the second or the fourth quadrant, choose $Q=2\pi s_m$ for a sufficiently large $m \in \mathbb{N}$ such that $H(Q/2\pi)$ is in the first quadrant, and thus $H(Q/2\pi) \notin H(K/2\pi)\mathbb{R}$.

Moreover, define $f(x) = \operatorname{arg}(H(K/2\pi) - H(x/2\pi))$ for $0 \leq x \leq Q$, which is a continuous function. Then, $f$ takes every value between the two extremal arguments. In other words, there exists an open set $V \subset \mathbb{S}^1$ such that for every $\mathbb{V} \in V$, there exists $Q_{\mathbb{V}} \in (0,Q)$ such that 
\[ \frac{H(K/2\pi)-H(Q_{\mathbb{V}}/2\pi)}{|H(K/2\pi)-H(Q_{\mathbb{V}}/2\pi)|} = \mathbb{V}. \]
Consequently, the method above shows that for every $V \in \mathbb{V}$ there is a sequence of approximations $\phi(t_{p_n,q_n}+Q_{\mathbb{V}}/(2\pi q_n^2))$ such that 
\[ \lim_{n \to \infty}{ \frac{ \phi(t_{p_n,q_n}+\frac{Q_{\mathbb{V}}}{2\pi q_n^2}) - \phi(t_{\rho}) }{|\phi(t_{p_n,q_n} + \frac{Q_{\mathbb{V}}}{2\pi q_n^2}) - \phi(t_{\rho})|} } = -\frac{1+i}{\sqrt{2}}\, e\, \mathbb{V}.  \]

\subsubsection{$\boldsymbol{H(K/2\pi) = 0}$}
Like before, take the subsequence of the approximations $p_n/q_n$ such that $e_n = e$ is constant. In this case, the limit is not \eqref{TangentFromApproximation1}, but on the other hand, from \eqref{PhiInTermsOfH} we get 
\begin{equation}\label{TheLimitIsZero}
\frac{\sqrt{2\pi}}{1+i} \, e^{-1} \,\lim_{n \to \infty}{  q_n^{3/2} \left( \phi(t_{\rho}) - \phi(t_{p_n,q_n}) \right) = \lim_{n \to \infty}H_n(K_n/2\pi) = 0.  } 
\end{equation}  
Choose $0 < Q < K/2$ as in the previous case so that 
\begin{equation*}
\begin{split}
&q_n^{3/2} \left( \phi(t_{p_n,q_n} + \textstyle{\frac{Q}{2\pi q_n^2}}) - \phi(t_{p_n,q_n}) \right) \\
&  =    q_n^{3/2} \left( \phi(t_{p_n,q_n} + \textstyle{\frac{Q}{2\pi q_n^2}}) - \phi(t_{\rho}) \right) +   q_n^{3/2} \left( \phi(t_{\rho}) - \phi(t_{p_n,q_n}) \right)
\end{split}
\end{equation*} 
and hence, by \eqref{TheLimitIsZero} we get
\begin{equation*}\label{LimitEquality}
\lim_{n \to \infty}{    q_n^{3/2} \left( \phi(t_{p_n,q_n} + \textstyle{\frac{Q}{2\pi q_n^2}}) - \phi(t_{\rho}) \right) } = \lim_{n \to \infty}{    q_n^{3/2} \left( \phi(t_{p_n,q_n} + \textstyle{\frac{Q}{2\pi q_n^2}}) - \phi(t_{p_n,q_n}). \right) }
\end{equation*}  
Then, if the previous limit does not vanish, the limit corresponding to these new approximations is
\begin{equation}\label{TangentForEveryQ}
\begin{split}
& \lim_{n \to \infty}{ \frac{ \phi(t_{p_n,q_n} + \frac{Q}{2\pi q_n^2}) - \phi(t_{\rho}) }{ \left| \phi(t_{p_n,q_n} + \frac{Q}{2\pi q_n^2}) - \phi(t_{\rho})  \right| } }  = \lim_{n \to \infty}{ \frac{ \phi(t_{p_n,q_n} + \frac{Q}{2\pi q_n^2}) - \phi(t_{p_n,q_n}) }{ \left| \phi(t_{p_n,q_n} + \frac{Q}{2\pi q_n^2}) - \phi(t_{p_n,q_n})  \right| } }  \\
& \qquad  = -\frac{1+i}{\sqrt{2}} e \lim_{n \to \infty}{ \frac{ H_n(Q/2\pi) }{ \left| H_n(Q/2\pi) \right| } } = -\frac{1+i}{\sqrt{2}} e \frac{ H(Q/2\pi) }{ \left| H(Q/2\pi) \right| }.
\end{split}
\end{equation} 
Then, since according to Lemma~\ref{DefOfTangentsLimits} the existence of a tangent requires that the limit above is the same for all approximations, it is enough to find two values of $Q$ giving different values for it. For that, using the sequences $s_m,\tilde{s}_m$ in \eqref{DefPointsForGoodSum}, choose $Q_1 = 2\pi s_m$ for big enough $m \in \mathbb{N}$ so that $H(Q_1/2\pi)$ is in the first quadrant, and $Q_2 = 2\pi \tilde{s}_m$ so that $H(Q_2/2\pi)$ is in the fourth quadrant. Then, 
\[ \frac{H(Q_1/2\pi)}{|H(Q_1/2\pi)|} \neq \frac{H(Q_2/2\pi)}{|H(Q_2/2\pi)|}. \]
 We can also make $0 < Q_1 < Q_2$.

Moreover, defining $f(x) = \operatorname{arg}H(x/2\pi)$, because $f(Q_2) < 0 < f(Q_1)$ and by continuity of $f$, there exists an open set $V \subset \mathbb{S}^1$ (corresponding to arguments $(f(Q_2),f(Q_1)) \subset [-\pi,\pi))$ such that for any $\mathbb{V} \in V$, there exists $Q_{\mathbb{V}} \in (Q_1,Q_2)$ such that 
\[ \frac{ H(Q_{\mathbb{V}}/2\pi) }{ \left| H(Q_{\mathbb{V}}/2\pi) \right| } = \mathbb{V}. \]
Consequently, for each $\mathbb{V} \in V$, there is a sequence of approximations $\phi(t_{p_n,q_n} + Q_{\mathbb{V}}/(2\pi q_n^2))$ so that 
\[  \lim_{n \to \infty}{ \frac{ \phi(t_{p_n,q_n}+\frac{Q_{\mathbb{V}}}{2\pi q_n^2}) - \phi(t_{\rho}) }{|\phi(t_{p_n,q_n} + \frac{Q_{\mathbb{V}}}{2\pi q_n^2}) - \phi(t_{\rho})|} } = -\frac{1+i}{\sqrt{2}}\, e \, \mathbb{V} .\]

\subsection{$\boldsymbol{K=0}$}
In this case, $ \lim_{n \to 0}H_n(K_n/2\pi) = H(0) = 0$.
\begin{figure}[h]
\centering
\includegraphics[width=0.85\linewidth]{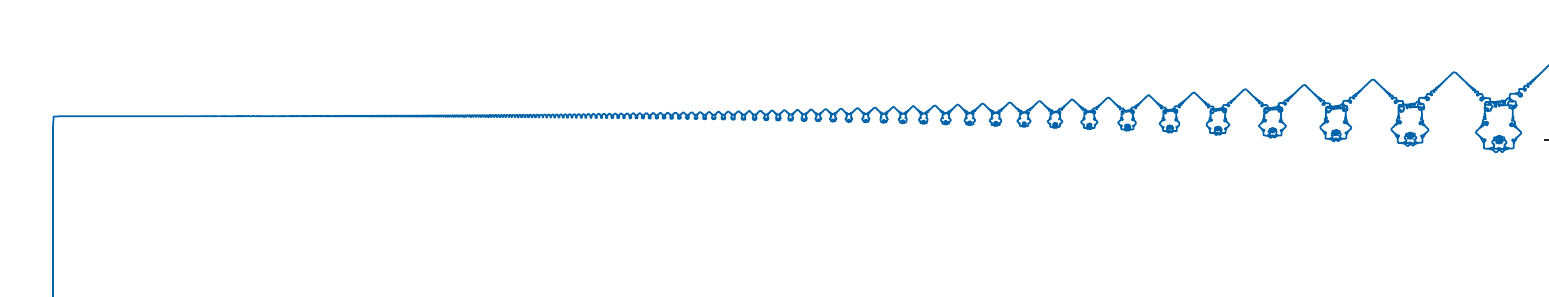}
\caption{\small A picture of the rescaled asymptotic $H(s)$, where the corner represents $H(0)$ and corresponds to $\phi(t_{p_n,q_n})$. When $K=0$, the irrational $\phi(t_\rho)$ approaches the corner. It is evident that at this scale any approximation to $t_\rho$ lying between itself and $t_{\rho_n}$ leads to the same tangent direction.}
\label{FIG_AsymptoticPicture}
\end{figure}
Thanks to the first expression in \eqref{DefOfHn}, we write
\[ \lim_{n \to \infty}{ \frac{H_n(K_n/2\pi)}{(K_n/2\pi)^{1/2}} } = \lim_{n \to \infty}   \frac{ \sqrt{K_n/2\pi} + 4i \, Y(\beta_n(K_n/2\pi)) \left(K_n/2\pi\right)^{3/2} + O(K_n^{5/2})  }{(K_n/2\pi)^{1/2}}    = 1. \]
As a consequence, 
\begin{equation}\label{DirectionCollapsed}
 \lim_{n \to \infty}{ \frac{H_n(K_n/2\pi)}{|H_n(K_n/2\pi)|} } = \lim_{s \to 0}{ \frac{ \frac{H_n(K_n/2\pi)}{(K_n/2\pi)^{1/2}} }{ \frac{|H_n(K_n/2\pi)|}{(K_n/2\pi)^{1/2}}  }   } = 1,
\end{equation}
and the limit in \eqref{ConditionTangent} is $  -  \lim_{n \to \infty}{ (1+i)e_n/\sqrt{2} }$. As before, to get the open set we are looking for, we take a subsequence of the approximations $p_n/q_n$ such that $e_n = e$ is constant for every $n \in \mathbb{N}$. The limit is then $ - e\,(1+i)/\sqrt{2} $.

It is obvious that we cannot proceed as in the case $K>0$, because every approximation taken at the scale of $H_n$ as in \eqref{TangentFromApproximation2} collapses to the same direction as in \eqref{DirectionCollapsed}. In the setting of Figure~\ref{FIG_AsymptoticPicture}, we approach more and more to the origin and \textit{we see nothing} when $n \to \infty$. A way to solve this is using the self-similarity term in \eqref{DefOfHn}, 
\begin{equation}\label{SelfSimilarCreator}
\frac16 - \frac{i}{2\pi}\frac{c_n}{q_n} - 2\phi\left( \frac{-1}{16\pi^2\beta_n(s)} \right).
\end{equation} 
Since $\phi(t + 1/2\pi) = \phi(t) + i/2\pi$ for every $t \in \mathbb{R}$, \eqref{SelfSimilarCreator} generates infinitely many copies of $F$  when $s\to 0$ because $1/\beta_n(s) \to \infty$. We will see that the more $s$ approaches to zero, the more precise copies of $F$ we get. Yet, they are also smaller, so a second rescaling will be needed.  Let us first identify the ranges in which copies are formed.

\begin{lem}\label{LemmaCorners}
Let $n \in \mathbb{N}$. The self-similar expression \eqref{SelfSimilarCreator} generates a copy of $F$ starting at each point
\begin{equation}\label{PrincipalCorners}
s_{n,m} = \frac{1}{2\pi}\frac{1}{4m - 2\frac{c_n}{q_n}}, \qquad \text{ for large enough } m \in \mathbb{N}.
\end{equation}
Moreover, $s_{n,m} \to 0$ implies $m \to \infty$, and in this situation, $s_{n,m} \simeq 1/(8\pi m )$ are equivalent infinitesimals.
\end{lem} 
\begin{proof}
Copies of $F$ are generated by $\phi$ starting at $m/2\pi$ for every $m \in \mathbb{Z}$, so we look for points $s >0$ such that
\[ \frac{-1}{16\pi^2\beta_n(s)} = - \frac{m}{2\pi}, \qquad \forall m \in \mathbb{N}. \]
Solving, we find that those points, which we call $s = s_{n,m}$, satisfy
\[ \frac{1}{8\pi m} = \beta_n(s) = \frac{s}{1 + 4\pi s \frac{c_n}{q_n}} \qquad   \Leftrightarrow \qquad s = s_{n,m} = \frac{1}{2\pi}\frac{1}{4 m - 2 \frac{c_n}{q_n}}. \]
Moreover, since $1 \leq c_n/q_n \leq 4$, then $s_{n,m} \to 0$ if and only if $m \to \infty$, and one can check that  $\lim_{m \to \infty} 8\pi m \, s_{n,m} = 1$.
\end{proof}

Now, we locate $K_n$ in its corresponding range described by the points \eqref{PrincipalCorners} in Lemma~\ref{LemmaCorners}. 
\begin{lem}\label{LemmaRangeOfCopiesOfF}
Let $n \in \mathbb{N}$ be large enough. Then, there exists a unique $m = m(n) \in \mathbb{N}$ such that
\[ s_{n, m(n)+1} < \frac{K_n}{2\pi} \leq s_{n, m(n)}. \]
Moreover, $\lim_{n \to \infty}m(n) = +\infty$. 
\end{lem}
\begin{proof}
For fixed $n \in \mathbb{N}$, the sequence $s_{n,m}$ is strictly decreasing to zero in $m$. Also since $\lim_{n \to \infty}{K_n} = 0$, choose $n$ big enough so that $K_n$ is smaller than at least one value $s_{n,m}$. Then, 
\[ \exists ! m = m(n) \qquad \text{ such that } \qquad  s_{n,m(n)+1} < \frac{K_n}{2\pi} \leq s_{n,m(n)}.  \]
As a consequence, since $\lim_{n \to \infty}K_n = 0$, we see that $ \lim_{n \to \infty}{ s_{n,m(n)+1} } = 0 $. 
According to Lemma~\ref{LemmaCorners}, this is equivalent to saying that $\lim_{n \to \infty}{m(n)+1} = +\infty$, which yields the result. 
\end{proof}

\begin{rem}
Many times, we write just $s_{m(n)}$ instead of $s_{n,m(n)}$.
\end{rem}

We proceed as follows: for fixed $n \in \mathbb{N}$, and once rescaled to $H_n$ by $q_n^{3/2}$, the point under analysis is in the interval $(s_{m(n)+1},s_{m(n)})$, which corresponds to a copy of $F$. Then, we rescale the function again when the parameter is in a wider interval, $(s_{m(n)+2},s_{m(n)})$, corresponding to two successive copies of $F$. While $\phi(t_{\rho})$ is in the first one, we take as approximations points in the second one, so that directions do not match to the ones of approximations $\phi(t_{\rho_n})$.

\begin{figure}[h]
\centering
\includegraphics[width=0.85\linewidth]{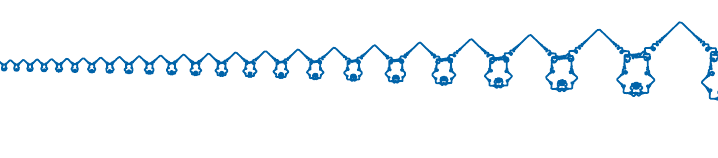}
\caption{\small Second rescaling of the asymptotic shown in Figure~\ref{FIG_AsymptoticPicture}, where the corner $H(0)$ is on the left, out of the picture. Self-similarity is clearly visible in terms of copies of $F$, all very similar to each other. The parameter interval $(s_{m(n)+2},s_{m(n)})$ represents two successive copies; the irrational $\phi(t_{\rho})$ lies on the right hand-sided one. }
\end{figure}

\subsubsection{\textbf{Step 1 - First rescaling to get $\boldsymbol{H_n}$}}
This step corresponds to \eqref{PhiInTermsOfH}.

\subsubsection{\textbf{Step 2 - Translate the copies of $\boldsymbol{F}$ in $\boldsymbol{(s_{n,m(n)+2},s_{n,m(n)})}$ to the origin}}
We translate a point on the left of $\phi(t_{\rho})$ to the origin. Since $\phi(t_{\rho})$ corresponds to $K_n \in (s_{m(n)+1},s_{m(n)})$, the approximation we use is $s_{m(n)+\mu}$ for some $1 \leq \mu \leq 2$, so  we compute
\begin{equation}\label{RescalingStep2}
H_n(s) - H_n(s_{m(n)+\mu})
\end{equation}  
using the asymptotic \eqref{DefOfHn}. If $s \in (s_{m(n)+1},s_{m(n)}]$, there exists $\alpha \in [0,1)$ such that
\begin{equation}\label{sAlpha}
 s = s_{m(n) + \alpha} = \frac{1}{2\pi} \frac{1}{4(m(n) + \alpha) - 2c_n/q_n}. 
\end{equation}
In particular, there exists $\alpha_n \in [0,1)$ such that $K_n = s_{m(n)+\alpha_n}$, and $\phi(t_{\rho})$ corresponds to $s_{m(n) + \alpha_n}$. Using \eqref{PrincipalCorners} and Taylor's expansion $(1+x)^{-1/2} \simeq 1 - x/2$ when $x \to 0$, the first term in the asymptotic of \eqref{RescalingStep2} is 
\begin{equation}\label{FirstTerm}
\begin{split}
\sqrt{s} - \sqrt{s_{m+\mu}} & = \frac{1}{\sqrt{2\pi}}\left( \frac{1}{\sqrt{4m + 4\alpha - 2\frac{c_n}{q_n}}} - \frac{1}{\sqrt{4m + 4\mu - 2\frac{c_n}{q_n}}} \right)\\
& \simeq \frac{1}{\sqrt{2\pi}}\,\frac{1}{\sqrt{4m}} \left( 1 - \frac{4\alpha - 2\frac{c_n}{q_n}}{8m} - 1 + \frac{4\mu - 2\frac{c_n}{q_n}}{8m}   \right)  = \frac{1}{\sqrt{2\pi}}\,\frac{\mu - \alpha}{4 m ^{3/2}}. 
\end{split}
\end{equation} 
when $m \to \infty$, where $\simeq$ stands for equivalent infinitesimals. The second term of \eqref{RescalingStep2} is
\[ \frac{8\pi^2}{3}i\left( \frac16 - \frac{i}{2\pi}\frac{c_n}{q_n} - 2\phi\left( \frac{-1}{16\pi^2\beta_n(s)} \right) \right)s^{3/2} - \frac{8\pi^2}{3}i\left( \frac16 - \frac{i}{2\pi}\frac{c_n}{q_n} - 2\phi\left( \frac{-1}{16\pi^2\beta_n(s_{m+\mu})} \right) \right) s_{m+\mu}^{3/2}, \]
which we split it into a sum $A+B$ such that
\begin{equation*}
\begin{split}
A & = \frac{8\pi^2}{3}i \left[ \left( \frac16 - \frac{i}{2\pi}\frac{c_n}{q_n} - 2\phi\left( \frac{-1}{16\pi^2\beta_n(s)} \right) \right) - \left( \frac16 - \frac{i}{2\pi}\frac{c_n}{q_n} - 2\phi\left( \frac{-1}{16\pi^2\beta_n(s_{m+\mu})} \right) \right)  \right] s^{3/2}  \\
& = \frac{8\pi^2}{3}i \left[ 2\phi\left( \frac{-1}{16\pi^2\beta_n(s_{m+\mu})} \right)  - 2\phi\left( \frac{-1}{16\pi^2\beta_n(s)}  \right)  \right] s^{3/2}
\end{split}
\end{equation*} 
and 
\[ B = \frac{8\pi^2}{3}i\left( \frac16 - \frac{i}{2\pi}\frac{c_n}{q_n} - 2\phi\left( \frac{-1}{16\pi^2\beta_n(s_{m+\mu})} \right) \right) \left( s^{3/2} - s_{m+\mu}^{3/2} \right). \]
From the proof of Lemma~\ref{LemmaCorners}, the definition of $s_{m + \alpha}$ and the periodic property of $\phi$, we have
\[ \phi\left( \frac{-1}{16\pi^2\beta_n(s_{m+\alpha})}\right) = \phi\left( -\frac{m+\alpha}{2\pi}  \right) = -i\frac{m+2}{2\pi} + \phi \left( \frac{2 - \alpha}{2\pi} \right), \]
so 
\[ A = \frac{16 \pi^2}{3} i \, \left(  \phi \left( \frac{2 - \mu}{2\pi} \right) - \phi \left( \frac{2 - \alpha}{2\pi} \right) \right) s^{3/2}. \]
On the other hand, in the same spirit as before, since $(1 + x)^{-3/2} \simeq 1 - 3x/2$ when $x \to 0$, we get
\begin{equation*}
\begin{split}
(2\pi)^{\frac32}\left( s^{3/2} - s_{m+\mu}^{3/2}\right) & = \frac{1}{ (4m + 4\alpha - 2\frac{c_n}{q_n})^{3/2} } - \frac{1}{ (4m + 4\mu - 2\frac{c_n}{q_n})^{3/2} } \\
& \simeq \frac{1}{(4m)^{3/2}} \left( 1 - \frac32 \frac{4\alpha - 2\frac{c_n}{q_n}}{4m} - 1 + \frac32 \frac{4\mu - 2\frac{c_n}{q_n}}{4m}  \right) = \frac{3(\mu - \alpha)}{16m^{5/2}} 
\end{split}
\end{equation*}
when $m \to \infty$. Hence, 
\begin{equation*}
\begin{split}
B & \simeq \frac{8\pi^2}{3}i\left( \frac16 - \frac{i}{2\pi}\frac{c_n}{q_n} + i\frac{m+2}{\pi} - 2\phi\left( \frac{2-\mu}{2\pi} \right) \right) \frac{3(\mu-\alpha)}{(2\pi)^{\frac32}16m^{5/2}} \\
&  = \frac{8\pi^2}{3}i\left( \frac16 - \frac{i}{2\pi}\frac{c_n}{q_n} + \frac{2i}{\pi} - 2\phi\left( \frac{2-\mu}{2\pi} \right) \right) \frac{3(\mu-\alpha)}{(2\pi)^{\frac32}16m^{5/2}} - \frac{(\mu-\alpha)}{4(2\pi)^{\frac12}m^{3/2}}.
\end{split}
\end{equation*}  
The last term in $B$ gets cancelled with \eqref{FirstTerm}, so
\begin{equation}\label{RescalingStep2Long}
\begin{split}
H_n(s) & - H_n(s_{m(n)+\mu})  \\
& = \frac{16 \pi^2}{3} i \, \left(  \phi \left( \frac{2 - \mu}{2\pi} \right) - \phi \left( \frac{2 - \alpha}{2\pi} \right) \right) s^{3/2} \\
&  \qquad + \frac{8\pi^2}{3}i\left( \frac16 - \frac{i}{2\pi}\frac{c_n}{q_n} + \frac{2i}{\pi} - 2\phi\left( \frac{2-\mu}{2\pi} \right) \right) \frac{3(\mu-\alpha)}{(2\pi)^{\frac32}16m^{5/2}}  + O(s^{\frac52}) + O(s_{m+\mu}^{\frac52}). 
\end{split}
\end{equation}
The rescaling suggested by \eqref{RescalingStep2Long} to obtain the copies of $F$ is $s^{-3/2}$.

\subsubsection{\textbf{Step 3 - Rescale the $\boldsymbol{s^{3/2}}$ term to identify the copy of $\boldsymbol{F}$}}
Define 
\begin{equation}\label{DefOfGn}
G_{n,\mu}(s) = s_{m+\mu}^{-\frac32}\left[ H_n(s) - H_n(s_{m+\mu}) \right], \qquad \text{ for } s \in (s_{m+1},s_m), 
\end{equation} 
which corresponds rescaling \eqref{RescalingStep2Long} with $s_{m+\mu}^{-3/2}$. First, we see that when $n \to \infty$ (and thus $m =m(n) \to \infty$), the higher order terms tend to zero independently of $\alpha$ because $\mu - \alpha \leq 2$ and
\[ 0 \leq \frac{\mu-\alpha}{m^{5/2}} s_{m+\mu}^{-3/2} \leq 2(2\pi)^{3/2}  \frac{ (4m + 4\mu - 2\frac{c_n}{q_n})^{3/2} }{ m^{5/2}  }  \to 0. \]
Also
\[ \frac{s_{m+\mu}^{5/2}}{s_{m+\mu}^{3/2}} = s_{m+\mu} \to 0 \qquad \text{ and } \qquad \frac{s^{5/2}}{s_{m+\mu}^{3/2}} = \frac{1}{2\pi}  \frac{ (4m + 4\mu - 2\frac{c_n}{q_n})^{3/2} }{ (4m + 4\alpha - 2\frac{c_n}{q_n})^{5/2} }  \to 0. \]
Consequently, 
\begin{equation}
\begin{split}
&  \lim_{n  \to \infty}{\left| G_{n,\mu}(s) - \frac{16 \pi^2}{3} i \, \left(  \phi \left( \frac{2 - \mu}{2\pi} \right) - \phi \left( \frac{2 - \alpha}{2\pi} \right) \right) \right| }\\
 & \qquad  = \lim_{n \to \infty}{ \left| \frac{16 \pi^2}{3} i \, \left(  \phi \left( \frac{2 - \mu}{2\pi} \right) - \phi \left( \frac{2 - \alpha}{2\pi} \right) \right) \left( \frac{s^{3/2}}{s_{m + \mu}^{3/2}} - 1  \right) \right|  }  \\
 & \qquad \leq \lim_{n \to \infty}{ C \left| \frac{s^{3/2}}{s_{m + \mu}^{3/2}} - 1 \right| } = 0,
\end{split}
\end{equation}
where $C>0$ is independent of $\alpha$ and $\mu$. The last equality holds because recalling that $s_{m+\mu} \leq s_{m+\alpha} = s$, we get
\[ 0 \leq \frac{s^{3/2}}{s_{m + \mu}^{3/2}} - 1 = \frac{4m + 4\mu - 2\frac{c_n}{q_n}}{4m + 4\alpha - 2\frac{c_n}{q_n}} - 1 = \frac{4(\mu - \alpha)}{4m + 4\alpha - 2\frac{c_n}{q_n} } \leq \frac{8}{4m + 4\alpha - 2\frac{c_n}{q_n} } \to 0. \]
Convergence is thus independent of $\alpha$, so we have proved the following proposition.

\begin{prop}\label{PropositionOfCopies}
Let $1 \leq \mu \leq 2$. Then,
\[ \lim_{n \to \infty}{ \sup_{\alpha \in (0,1)} \left| G_{n,\mu}\left( \frac{ 1 }{ 4m(n) + 4\alpha - 2\frac{c_n}{q_n} } \right) - \frac{16 \pi^2}{3} i \, \left(  \phi \left( \frac{2 - \mu}{2\pi} \right) - \phi \left( \frac{2 - \alpha}{2\pi} \right) \right) \right| } = 0. \]
\end{prop}

\subsubsection{\textbf{Step 4 - Conclusion}}
Let us write $G_{n,\mu}$ in terms of $\phi$ using \eqref{Asymptotic013Rescaled}, \eqref{PhiInTermsOfH},  and \eqref{DefOfGn} so that
\[ G_{n,\mu}(s) = \frac{\sqrt{2\pi}}{1+i} e^{-1}q_n^{3/2}s_{m(n)+\mu}^{-3/2}\left[ \phi\left( t_{\rho_n} + \frac{s}{q_n^2} \right) - \phi\left( t_{\rho_n} + \frac{s_{m(n)+\mu}}{q_n^2} \right)  \right].  \]
Recall that $K_n/2\pi = s_{m(n) + \alpha_n}$, define $\alpha = \liminf_{n\to\infty}{\alpha_n} \in [0,1]$ and take the subsequence which, after being relabelled, satisfies $\lim_{n \to \infty}{\alpha_n} = \alpha$. From Proposition~\ref{PropositionOfCopies} and the continuity of $\phi$, we can write
\[ \lim_{ n \to \infty }{ G_{n,\mu}(K_n/2\pi) } = \frac{16 \pi^2}{3} i \, \left(  \phi \left( \frac{2 - \mu}{2\pi} \right) - \phi \left( \frac{2 - \alpha}{2\pi} \right) \right) \]
Then, since as long as $\alpha < \mu$ we have $0 < s_{m(n)+\mu} < s_{m(n)+\alpha}$, let us use the points $t_{\rho_n} + s_{m(n)+\mu}/q_n^2$ as approximations to $t_{\rho}$ for the limit defining the tangent, which is
\begin{equation}\label{TangentFromApproximation3}
\begin{split}
  \lim_{n \to \infty}{  \frac{ \phi\left( t_{\rho} \right) - \phi\left( t_{\rho_n} + \frac{s_{m(n)+\mu}}{q_n^2} \right)  }{ \left| \phi\left( t_{\rho} \right) - \phi\left( t_{\rho_n} + \frac{s_{m(n)+\mu}}{q_n^2} \right) \right|  } }  
& =  \lim_{n \to \infty}{ \frac{  \frac{1+i}{\sqrt{2\pi}}\,e \, q_n^{-3/2}s_{m(n)+\mu}^{3/2} G_{n,\mu}(K_n/2\pi)  }{  \left| \frac{1+i}{\sqrt{2\pi}} \, e \, q_n^{-3/2}s_{m(n)+\mu}^{3/2} G_{n,\mu}(K_n/2\pi) \right| }   } \\
& =  \frac{1+i}{\sqrt{2}} \, e \, \lim_{n \to \infty}{ \frac{    G_{n,\mu}(K_n/2\pi)  }{  \left|  G_{n,\mu}(K_n/2\pi) \right| }   }  \\
&  =   i \frac{1+i}{\sqrt{2}} \, e \, \frac{     \phi \left( \frac{2 - \mu}{2\pi} \right) - \phi \left( \frac{2 - \alpha}{2\pi} \right)    }{  \left|    \phi \left( \frac{2 - \mu}{2\pi} \right) - \phi \left( \frac{2 - \alpha}{2\pi} \right)   \right|  }  .
\end{split}
\end{equation}
Define 
\[ f(x) = \operatorname{arg} \left( \phi\left( \frac{2 - x}{2\pi} \right) - \phi\left(\frac{2 - \alpha}{2\pi} \right)  \right), \qquad \text{ for } x \in (1,2), \]
which is continuous. Moreover, $f(1) =  \operatorname{arg} \left( \frac{i}{2\pi} - \phi\left(\frac{2 - \alpha}{2\pi} \right)  \right)$ and $f(2) = \operatorname{arg} \left(  - \phi \left(\frac{2 - \alpha}{2\pi} \right)  \right)$. If $\alpha \in (0,1)$, then $f(1) \neq f(2)$ and by continuity, $f$ takes all values in $( f(1),f(2)) \in [0,2\pi)$. If $\alpha = 0$ or $\alpha = 1$, we may choose the interval $(f(2),f(3/2))$. In both cases, these intervals correspond to an open set $V \subset \mathbb{S}^1$ such that for any $\mathbb{V} \in V$, there exists $\mu_{\mathbb{V}} \in (1,2)$ such that 
\[ \frac{     \phi \left( \frac{2 - \mu_{\mathbb{V}}}{2\pi} \right) - \phi \left( \frac{2 - \alpha}{2\pi} \right)    }{  \left|    \phi \left( \frac{2 - \mu_{\mathbb{V}}}{2\pi} \right) - \phi \left( \frac{2 - \alpha}{2\pi} \right)   \right|  } = \mathbb{V}, \]
and therefore there is a sequence of approximations $\phi(t_{\rho_n} + s_{m(n)+\mu_{\mathbb{V}}}/q_n^2) $ such that 
\[   \lim_{n \to \infty}{  \frac{ \phi\left( t_{\rho} \right) - \phi\left( t_{\rho_n} + s_{m(n)+\mu_{\mathbb{V}}}/q_n^2 \right)  }{ \left| \phi\left( t_{\rho} \right) - \phi\left( t_{\rho_n} + s_{m(n)+\mu_{\mathbb{V}}}/q_n^2 \right) \right|  } }  =  i \frac{1+i}{\sqrt{2}} \, e \, \mathbb{V}. \]

\appendix

\section{Proof of Lemma~\ref{LemmaUniformConvergence}}\label{Appendix}
\renewcommand\thesection{\arabic{section}}
\setcounter{section}{5}
\setcounter{thm}{4}
\begin{lem}
Let $M >0$. Then, the sequence of functions $H_n$ converges uniformly to $H$ in $(0,M)$. In other words,
\[ \lim_{n \to \infty}{\lVert H_n - H \rVert_{L^{\infty}([0,M])}} = 0. \]
\end{lem}
To prove this lemma, we need some preliminary computations. Remember that we are working with an irrational $\rho \in [0,1]\setminus\mathbb{Q}$ that is approximated by the continued fraction convergents $p_n/q_n$. The sequence $c_n$ corresponds to the values of $c_\pm$ appearing in Proposition~\ref{PropositionFromPart1} for each of the approximations $p_n/q_n$. Also, recall that several subsequences were extracted such that $c_n >0$ and also such that $\lim_{n\to\infty}c_n/q_n = a$. 
\renewcommand\thesection{\Alph{section}}
\setcounter{section}{1}
\setcounter{thm}{0}
\begin{lem}\label{AuxiliaryLemma}
Let $\phi$ be Riemann's non-differentiable function \eqref{RiemannFunction}. Let $\beta_n(s)$ and $\beta(s)$ be defined as in \eqref{DefOfBetan} and \eqref{DefOfBeta} and $a$ as in \eqref{DefOfA}.  Let $M >0$. Then, there exists a constant $C>0$ such that the following hold for any $s>0$ and for any $n \in \mathbb{N}$:
\begin{enumerate}[(a)]
	\item \label{itm:1} $ 0 \leq \beta_n(s) \leq 1/4\pi $ and $  0 \leq \beta(s) \leq 1/4\pi $.
	\item \label{itm:2} If $s \leq M$, then $1 \leq (1-4\pi\frac{c_n}{q_n}\beta_n(s))^{-1} \leq 1 + 16\pi M$ \text{ and } $1 \leq (1-4\pi a \beta(s))^{-1} \leq 1 + 16\pi M$.
	\item \label{itm:3} $ \left|\beta_n(s) - \beta(s) \right| \leq C\, \left| a - c_n/q_n \right| $.  
	\item \label{itm:4} $ \left| \phi( \beta_n(s)) - \phi(\beta(s)) \right| \leq C\, \left| a - c_n/q_n \right|^{1/2} $.
	\item \label{itm:5} $ \left|\beta_n(s)^2 - \beta(s)^2 \right| \leq C\, \left| a - c_n/q_n \right|$. 
	\item \label{itm:6} $ \left|\beta_n(s)^3 - \beta(s)^3 \right| \leq C\, \left| a - c_n/q_n \right| $. 
	\item \label{itm:7} $ \left| \frac{c_n}{q_n}\beta_n(s) - a\beta(s) \right| \leq C\, \left| a - c_n/q_n \right| $. 
	\item \label{itm:8} $ \left| \left(\frac{c_n}{q_n}\beta_n(s) \right)^2 - \left(a\beta(s)\right)^2 \right| \leq C\, \left| a - c_n/q_n \right| $. 
		\item \label{itm:9} $ \left| \left(\frac{c_n}{q_n}\beta_n(s) \right)^3 - \left(a\beta(s)\right)^3 \right| \leq   C\, \left| a - c_n/q_n \right| $. 
\end{enumerate}
\end{lem}
\begin{proof}
For \eqref{itm:1}, using that $s \geq 0$ and $1 \leq c_n/q_n \leq 4$, we get 
\[ 0 \leq \beta_n(s) =  \frac{s}{1 + 4\pi \frac{c_n}{q_n}s} \leq \frac{s}{1+4\pi s} \leq \frac{1}{4\pi}.  \]
The same holds for $\beta(s)$ because $1 \leq a \leq 4$. For \eqref{itm:2} by the definition of $\beta_n(s)$ and if $0 \leq s \leq M$, we write 
\[ 1-4\pi \frac{c_n}{q_n}\beta_n(s) =  \frac{1}{1+4\pi \frac{c_n}{q_n}s} \quad \Longrightarrow \quad 1 \leq  \frac{1}{1-4\pi\frac{c_n}{q_n}\beta_n(s)} = 1+4\pi \frac{c_n}{q_n}s \leq  1 + 16\pi M. \]
The other inequality in \eqref{itm:2} follows the same way. For \eqref{itm:3}, we write
\[ \beta_n(s) - \beta(s) = \frac{s}{1 + 4\pi \frac{c_n}{q_n}s} - \frac{s}{1 + 4\pi a s}  =  \left( a  - \frac{c_n}{q_n} \right) \frac{4\pi s^2 }{(1 + 4\pi \frac{c_n}{q_n} s)(1 + 4\pi a s)} =  4\pi \beta(s)\beta_n(s) \left( a  - \frac{c_n}{q_n} \right),\]
and using \eqref{itm:1}, we get
\[ \left|\beta_n(s) - \beta(s) \right|  \leq \frac{1}{4\pi} \left| a  - \frac{c_n}{q_n} \right| .\]
To prove \eqref{itm:4}, we need a result of Duistermaat \cite[Lemma 4.1]{Duistermaat1991} stating that $\phi_D$ is globally $C^{1/2}$. From this, there exists $C>0$ such that 
\[ \left| \phi(x)-\phi(y) \right| \leq C\left( |x-y|^{1/2} + |x-y| \right), \quad \forall x,y \in \mathbb{R}.  \]
This turns into $  \left| \phi(x)-\phi(y) \right| \leq C |x-y|^{1/2} $ in case $|x-y| \leq 1$. Then, by \eqref{itm:3}, we have $|\beta_n(s) - \beta(s)| \leq 1$, so 
\[ \left| \phi( \beta_n(s)) - \phi(\beta(s)) \right| \leq C\left|\beta_n(s) - \beta(s) \right|^{1/2}  \leq C\, \left| a  - \frac{c_n}{q_n} \right|^{1/2}. \]
Properties \eqref{itm:5} and \eqref{itm:6} are proved like \eqref{itm:3}. Indeed, 
\begin{equation*}
\begin{split}
\beta_n(s)^2 - \beta(s)^2 & = s^2\, \frac{(1 + 4\pi a s)^2 - (1 + 4\pi \frac{c_n}{q_n} s)^2 }{ (1 + 4\pi \frac{c_n}{q_n} s)^2(1 + 4\pi a s)^2 } \\
& = 8\pi\frac{s^3 \, \left( a - c_n/q_n \right)}{(1 + 4\pi a s)^2(1 + 4\pi \frac{c_n}{q_n} s)^2} + 16\pi^2\frac{s^4\, \left( a^2 - \left(c_n/q_n\right)^2 \right)}{(1 + 4\pi a s)^2(1 + 4\pi \frac{c_n}{q_n} s)^2}.
\end{split}
\end{equation*} 
Since $a,c_n/q_n \leq 4$ implies $a + c_n/q_n \leq 8$, using \eqref{itm:1} we get 
\begin{equation*}
\begin{split}
\left| \beta_n(s)^2 - \beta(s)^2 \right| & \leq 8\pi |\beta_n(s)|^2 \frac{|\beta(s)|}{1 + 4\pi a s}\left| a - \frac{c_n}{q_n} \right| + 16\pi^2|\beta_n(s)|^2|\beta(s)|^2\left| a - \frac{c_n}{q_n} \right|\left| a + \frac{c_n}{q_n} \right| \\
& \leq \frac{1}{8\pi^2} \left| a - \frac{c_n}{q_n} \right| + \frac{8}{16\pi^2}\left| a - \frac{c_n}{q_n} \right| = \frac{5}{8\pi^2}\left| a - \frac{c_n}{q_n} \right| .
\end{split}
\end{equation*}
For \eqref{itm:6}, we write
\begin{equation*}
\begin{split}
\beta_n(s)^3 - \beta(s)^3 & = s^3 \frac{(1 + 4\pi a s)^3 - (1 + 4\pi \frac{c_n}{q_n} s)^3}{ (1 + 4\pi \frac{c_n}{q_n} s)^3(1 + 4\pi a s)^3 } \\
& = 12\pi s^4\frac{a-c_n/q_n}{(1 + 4\pi \frac{c_n}{q_n} s)^3(1 + 4\pi a s)^3} + 48\pi^2 s^5\frac{a^2-\left(c_n/q_n\right)^2}{(1 + 4\pi \frac{c_n}{q_n} s)^3(1 + 4\pi a s)^3} \\
& + (4\pi)^3 s^6\frac{a^3-\left(c_n/q_n\right)^3}{(1 + 4\pi \frac{c_n}{q_n} s)^3(1 + 4\pi a s)^3} ,
\end{split}
\end{equation*}
so since $|a^3-\left(c_n/q_n\right)^3| \leq  |a^2 + ac_n/q_n + c_n^2/q_n^2|  \leq 3\cdot 4^2 |a-c_n/q_n| $, we get
\begin{equation*}
\begin{split}
\left| \beta_n(s)^3 - \beta(s)^3 \right| & \leq 12\pi |\beta_n(s)|^3|\beta(s)|\left| a - c_n/q_n \right| + 48\pi^2 \cdot 8 |\beta_n(s)|^3|\beta(s)|^2\left| a - c_n/q_n \right| \\
& + (4\pi)^3|\beta_n(s)|^3|\beta(s)|^3 48 \left| a - c_n/q_n \right| \\
& \leq C\, \left| a - c_n/q_n \right|. 
\end{split}
\end{equation*}
The remaining properties are proved all by the same method. For \eqref{itm:7}, we use \eqref{itm:1} and \eqref{itm:3} to write
\[  \left| \frac{c_n}{q_n}\beta_n(s) - a\beta(s) \right| \leq |\beta_n(s)|\left| \frac{c_n}{q_n}-a \right| + a|\beta_n(s) - \beta(s)| \leq \frac{1+a}{4\pi} \left| a -  \frac{c_n}{q_n}  \right| \leq  \frac{5}{4\pi} \left| a -  \frac{c_n}{q_n}  \right|. \]
For \eqref{itm:8}, we use \eqref{itm:1} and \eqref{itm:5} and write
\begin{equation*}
\begin{split}
\left| \left(\frac{c_n}{q_n}\beta_n(s) \right)^2 - \left(a\beta(s)\right)^2 \right| &  \leq |\beta_n(s)|^2 \left| \left(\frac{c_n}{q_n} \right)^2 - a^2 \right| + a^2\left| \beta_n(s)^2 - \beta(s)^2 \right|  \leq C\, \left| a -  \frac{c_n}{q_n}\right|. \\
\end{split}
\end{equation*}
Last, for \eqref{itm:9}, we use \eqref{itm:1} and \eqref{itm:6} to write
\begin{equation*}
\begin{split}
\left| \left(\frac{c_n}{q_n}\beta_n(s) \right)^3 - \left(a\beta(s)\right)^3 \right| &  \leq |\beta_n(s)|^3 \left| \left(\frac{c_n}{q_n} \right)^3 - a^3 \right| + a^3\left| \beta_n(s)^3 - \beta(s)^3 \right| \leq C\, \left| a -  \frac{c_n}{q_n}\right| . \\
\end{split}
\end{equation*}
\end{proof}

\begin{proof}[Proof of Lemma~\ref{LemmaUniformConvergence}]
In this proof, we disregard constants not depending on $M$ and on $s$. From \eqref{DefOfHnClosed} and \eqref{DefOfH} we write
\begin{equation}\label{SplittingProofLemma}
H_n(s) - H(s) = A + B,
\end{equation}
where
\begin{equation}\label{ProofLemmaA}
A = \frac{\phi(\beta_n(s))}{(1-4\pi\frac{c_n}{q_n}\beta_n(s))^{3/2}} - \frac{\phi(\beta(s))}{(1-4\pi a \beta(s))^{3/2}}
\end{equation}
and 
\begin{equation}\label{ProofLemmaB}
B = \frac{c_n}{q_n} \int_0^{\beta_n(s)}{ \frac{\phi(r)}{(1-4\pi\frac{c_n}{q_n}r)^{5/2}}\,dr } - a \int_0^{\beta(s)}{ \frac{\phi(r)}{(1-4\pi  a r)^{5/2}}\,dr }.
\end{equation}
Let us split $A = A_1 + A_2$ such that
\[ A_1 = \phi(\beta_n(s)) \left( \frac{1}{(1-4\pi\frac{c_n}{q_n}\beta_n(s))^{3/2}} - \frac{1}{(1-4\pi a\beta(s))^{3/2}} \right) \]
and 
\[ A_2 =  \frac{\phi(\beta_n(s)) - \phi(\beta(s))}{\left( 1-4\pi a \beta(s)\right)^{3/2}}. \]
By Lemma~\ref{AuxiliaryLemma}-\eqref{itm:2} and \eqref{itm:4}, we have
\[ |A_{2}| \leq (1+16\pi M)^{\frac32} \left| a - \frac{c_n}{q_n} \right|^{\frac12}. \]
For $A_1$, by Lemma~\ref{AuxiliaryLemma}-\eqref{itm:1} we have $ |\phi(\beta_n(s))| \leq \lVert \phi \rVert_{L^{\infty}([0,1/(4\pi)]} $, so using Lemma~\ref{AuxiliaryLemma}-\eqref{itm:2} we write
\begin{equation*}\label{BoundA1}
\begin{split}
|A_1| & \leq \frac{\left| (1-4\pi\frac{c_n}{q_n}\beta_n(s))^3 - (1-4\pi a \beta(s))^3 \right| }{(1-4\pi\frac{c_n}{q_n}\beta_n(s))^{\frac32}(1-4\pi a \beta(s))^{\frac32}\left( (1-4\pi\frac{c_n}{q_n}\beta_n(s))^{\frac32} + (1-4\pi a \beta(s))^{\frac32} \right)}  \\
& \leq C_M\,   \left| (1-4\pi\frac{c_n}{q_n}\beta_n(s))^3 - (1-4\pi a \beta(s))^3 \right| \\
& \leq C_M\, \left( \left| \frac{c_n}{q_n}\beta_n(s) - a\beta(s) \right| + \left| \left(\frac{c_n}{q_n}\beta_n(s) \right)^2 - (a\beta(s))^2 \right| + \left| \left(\frac{c_n}{q_n}\beta_n(s) \right)^3 - (a\beta(s))^3 \right| \right),
\end{split}
\end{equation*} 
where $C_M = (1+16\pi M)^{9/2}$. Then, from Lemma~\ref{AuxiliaryLemma}-\eqref{itm:7}, \eqref{itm:8} and \eqref{itm:9} we get $ |A_1| \leq C_M  \left| a - c_n/q_n \right|$. Overall,
\begin{equation}\label{BoundForA}
 |A| \leq C_M \left( \left| a - \frac{c_n}{q_n} \right|^{1/2} + \left| a - \frac{c_n}{q_n} \right| \right).
\end{equation}
For $B$, we write
\begin{equation*}
\begin{split}
|B| & \leq \left| a - \frac{c_n}{q_n} \right| \left| \int_0^{\beta_n(s)}{ \frac{\phi(r)}{(1-4\pi\frac{c_n}{q_n}r)^{5/2}}\,dr } \right| + a \left| \int_0^{\beta_n(s)}{ \left( \frac{\phi(r)}{(1-4\pi\frac{c_n}{q_n}r)^{5/2}} -  \frac{\phi(r)}{(1-4\pi a r)^{5/2}} \right) \,dr } \right| \\
& + a \left| \int_{\beta(s)}^{\beta_n(s)}{ \frac{\phi(r)}{(1-4\pi a r)^{5/2}}\,dr } \right| \\
& = |B_1| + |B_2| + |B_3|.
\end{split}
\end{equation*}  
From Lemma~\ref{AuxiliaryLemma}-\eqref{itm:2}, we deduce that if $0 \leq r \leq  \beta_n(s)$, then
\begin{equation}\label{APropertyForR}
 1 \leq \frac{1}{1-4\pi\frac{c_n}{q_n}r} \leq 1+16\pi M.
\end{equation}  
Hence, 
\[ |B_1| \leq |\beta_n(s)| \lVert \phi \rVert_{L^{\infty}([0,1/(4\pi)])}(1+16\pi M)^{\frac52} \left| a - c_n/q_n \right| \leq C_M\, \left| a - c_n/q_n \right|. \]
Also from Lemma~\ref{AuxiliaryLemma}-\eqref{itm:2}, if $r$ is between $\beta_n(s)$ and $\beta(s)$, then $r \leq \max\{  \beta_n(s), \beta(s)\}$, so 
\begin{equation}\label{APropertyForR2}
 1 \leq \frac{1}{1-4\pi a r} \leq 1+16\pi M
\end{equation}
is also satisfied. Thus, by Lemma~\ref{AuxiliaryLemma}-\eqref{itm:3} we get
\[ |B_3| \leq a |\beta_n(s)-\beta(s)| \lVert \phi \rVert_{L^{\infty}([0,1/(4\pi)])}(1+16\pi M)^{\frac52} \leq C_M\, \left| a - c_n/q_n \right|.  \]
For $B_2$, we need to compute
\[ \frac{1}{(1-4\pi\frac{c_n}{q_n}r)^{\frac52}} - \frac{1}{(1-4\pi  a r)^{\frac52}} = \frac{(1-4\pi a r)^5 - (1-4\pi\frac{c_n}{q_n}r)^5}{ (1-4\pi\frac{c_n}{q_n}r)^{\frac52}(1-4\pi a r)^{\frac52}\left( (1-4\pi\frac{c_n}{q_n}r)^{\frac52} + (1-4\pi a r)^{\frac52} \right) },  \]
and from \eqref{APropertyForR} and \eqref{APropertyForR2}, we get
\begin{equation*}
\begin{split}
|B_2| & \leq a|\beta_n(s)|\lVert \phi \rVert_{L^{\infty}([0,1/(4\pi)])}(1+16\pi M)^{\frac{15}{2}} \left| (1-4\pi a r)^5 - (1-4\pi\frac{c_n}{q_n}r)^5 \right|  \\
& \leq C_M\, \sum_{k=1}^5 r^k\,|a^k-(c_n/q_n)^k|
\end{split}
\end{equation*}
maybe renaming $C_M$. 
Here, $0 \leq r \leq  \beta_n(s) \leq 1/(4\pi)$, and also there exists $C>0$ such that $|a^k - (c_n/q_n)^k| \leq C\,|a-c_n/q_n|$ for every $k = 1,2,3,4,5$. Consequently, $ |B_2| \leq C_M\,  \left| a-c_n/q_n \right|$ , so 
\[ |B| \leq C_M \,  \left| a-\frac{c_n}{q_n} \right|. \]
 Joining it with \eqref{BoundForA}, we get
\[  \left| H_n(s) - H(s) \right| \leq C_M \, \left(  \left| a-\frac{c_n}{q_n} \right|^{1/2} + \left| a-\frac{c_n}{q_n} \right| \right) , \qquad \forall s \in [0,M]. \]
Therefore, we get the result
\[ \lim_{n \to \infty}{ \lVert H_n - H \rVert_{L^{\infty}([0,M])} } \leq C_M \lim_{n \to \infty}{\left( \left| a-\frac{c_n}{q_n} \right|^{1/2} + \left| a-\frac{c_n}{q_n} \right| \right)} = 0. \]
\end{proof}

\bibliographystyle{acm}
\bibliography{RiemannGeometryPartII.bib}

\end{document}